\documentclass[11pt]{amsart}
 \usepackage{amsopn}
 \usepackage{amsmath,amsthm,amssymb}

 \newcommand{\nc}{\newcommand}

 \renewcommand{\aa}{\mathfrak{a} } \nc{\aff}{\mathfrak{aff} } \nc{\bb}{\mathfrak{b} }
 \nc{\cc}{\mathfrak{c} }  \nc{\dd}{\mathfrak{d} }
 \nc{\ggo}{\mathfrak{g} }
 \nc{\hh}{\mathfrak{h} }  \nc{\ii}{\mathfrak{i} }
 \nc{\jj}{\mathfrak{j} }  \nc{\kk}{\mathfrak{k} }
\nc{\mm}{\mathfrak{m} }   \nc{\nn}{\mathfrak{n} }
\nc{\pp}{\mathfrak{p} }  \nc{\rr}{\mathfrak{r} } \nc{\sg}{\mathfrak{s} }
 \nc{\sog}{\mathfrak{so} }  \nc{\spg}{\mathfrak{sp} }
 \nc{\sug}{\mathfrak{su} }  \nc{\slg}{\mathfrak{sl} }
 \nc{\tg}{\mathfrak{t} }  \nc{\uu}{\mathfrak{u} }
 \nc{\vv}{\mathfrak{v} } \nc{\ww}{\mathfrak{w} }
 \nc{\zz}{\mathfrak{z} }

 \nc{\ggob}{\overline{\mathfrak{g}}}

\nc{\glg}{\mathfrak{gl} }

\nc{\pca}{\mathcal{P}} \nc{\nca}{\mathcal{N}}

 \nc{\vp}{\varphi} \nc{\ddt}{\frac{{\rm d}}{{\rm d}t}}
 \nc{\la}{\langle} \nc{\ra}{\rangle}

 \nc{\SO}{{\sf SO}} \nc{\Spe}{{\sf Sp}} \nc{\Sl}{{\sf Sl}}
 \nc{\SU}{{\sf SU}} \nc{\Or}{{\sf O}} \nc{\U}{{\sf U}}
 \nc{\Gl}{{\sf Gl}} \nc{\Se}{{\sf S}} \nc{\Cl}{{\sf Cl}}
 \nc{\Spin}{{\sf Spin}} \nc{\Pin}{{\sf Pin}}

 \nc{\RR}{{\mathbb R}} \nc{\HH}{{\mathbb H}} \nc{\CC}{{\mathbb C}}
 \nc{\ZZ}{{\mathbb Z}} \nc{\FF}{{\mathbb F}} \nc{\NN}{{\mathbb N}}
 \nc{\GG}{{\mathbb G}} \nc{\JJ}{{\mathbb J}} \nc{\II}{{\mathbb I}}
 \nc{\KK}{{\mathbb K}} \nc{\DD}{{\mathbb D}}

 \nc{\ad}{\operatorname{ad}} \nc{\Ad}{\operatorname{Ad}}
 \nc{\coad}{\operatorname{coad}} \nc{\ct}{\operatorname{T}}
 \nc{\rank}{\operatorname{rank}} \nc{\Irr}{\operatorname{Irr}}
 \nc{\End}{\operatorname{End}} \nc{\Aut}{\operatorname{Aut}}
 \nc{\Inn}{\operatorname{Inn}} \nc{\Der}{\operatorname{Der}}

 \theoremstyle{plain}
 \newtheorem{thm}{Theorem}[section]
 \newtheorem{prop}[thm]{Proposition}
 \newtheorem{cor}[thm]{Corollary}
 \newtheorem{lem}[thm]{Lemma}

 \theoremstyle{definition}
 \newtheorem{defn}[thm]{Definition}

 \theoremstyle{remark}
 \newtheorem*{remark}{Remark}
 
 \newtheorem{example}[thm]{Example}

\begin{document}

\title[complex structures on  tangent and cotangent Lie algebras of dimension six]
{complex structures on  tangent and cotangent Lie algebras of dimension six}

\author{Rutwig Campoamor-Stursberg}
\address{R. Campoamor-Stursberg: IMI y FCM, Universidad Complutense de Madrid}
\thanks{R.C.S. was partially supported by the research projects MTM2006-09152
of the M.E.C. and CCG07-UCM/ESP-2922 of the C.M.}
\email{rutwig@pdi.ucm.es}


\author{Gabriela P. Ovando}
\address{G. Ovando: CONICET y ECEN-FCEIA, Universidad Nacional de Rosario}
\thanks{G.O. was partially supported by CONICET, ANPCyT, SECyT-UNC, SCyT-UNR}

\email{ovando@mate.uncor.edu}


\date{17/2/2009}

\dedicatory{}

\commby{}


\begin{abstract}
This paper deals with complex structures on   Lie algebras
$\ct_{\pi} \hh=\hh \ltimes_{\pi} V$, where $\pi$ is either the
adjoint or the coadjoint representation. The main topic is the existence
question of complex structures on $\ct_{\pi}  \hh$ for $\hh$ a three
dimensional  real Lie algebra. First it was proposed  the study of 
 complex structures $J$ satisfying the constrain $J\hh=V$. Whenever $\pi$ is the
 adjoint representation this kind of complex structures  are associated
to non singular derivations of $\hh$. This fact derives different kind of
 applications. Finally an approach to the pseudo  K\"ahler geometry was done. 
\end{abstract}

\maketitle

\section{Introduction}


\medskip

While the existence of complex structures on reductive Lie algebras of even dimension was solved in different steps (starting with \cite{Sm} and \cite{Wa}), the solvable case remains an open problem. 
For dimensions up
to four, complex structures were studied in \cite{Sa1,SJ,O}; in dimension six the classification and induced complex geometry has been considered in the
nilpotent case in
\cite{CFGU1,CFGU2,CFU,FPS,GR,Mg,Sl, Ug}. Since those works are mainly done on the basis  of  a case
by case study, one of the principal obstructions in classifying complex
(and more general) structures on solvable  Lie groups
of dimensions equal or greater than six relies in the high number
of isomorphism classes. This implies
that different criteria have to be developed in order to describe
any kind of  geometry on Lie groups. One of these is the notion of
generalized complex structure, introduced by Hitchin in \cite{Hi}
and treated by various authors (see for instance \cite{ABDF,CG,G}
and references therein). On the other hand, in order to study the complex geometry, special types of complex
structures were considered, the so called abelian \cite{BDM}
and nilpotent  \cite{CFGU1, CFGU2}, specific for  nilpotent Lie
algebras, and which have been shown to be of considerable
interest, in particular in combination with other compatible
geometric structures.

The aim of this work is the study of  complex structures   on
tangent  and cotangent Lie algebras, that is
 Lie algebras which are semidirect products $\ct_{\pi} \hh= \hh \ltimes_{\pi} V$, for $\dim
V=\hh$, originally for $\pi$ the adjoint or the coadjoint
representation. We focus in the following existence questions:

\vskip 3pt

i)  complex structures satisfying the condition $J\hh=V$,

ii)  complex structures on $\ct \hh$ and $\ct^*\hh$ where $\hh$ is a three dimensional real Lie algebra, 

iii) symplectic  structures which are compatible for a  complex structure in
  ii), therefore  inducing pseudo K\"ahler geometries.

\vskip 3pt

Complex structures appearing in i)  are called {\em totally real}. They have 
become objects of  importance in the construction of weak mirror pairs 
(see for instance \cite{CLP} and references therein).

Complex and symplectic geometry are extremal special cases of generalized 
complex geometry. Once a Lie algebra $\hh$ was fixed, the corresponding
underlying geometric structure arises either 
as a complex structure on $\hh$ or as a totally real complex structure on 
$\ct^*\hh$, which is Hermitian for the canonical metric on $\ct^*\hh$.

 For the adjoint representation  we prove that  a totally real complex 
 structure corresponds to a non singular derivation of $\hh$. Therefore the
 existence of such a totally real complex structure on $\ct \hh$ imposes the
 condition on  $\hh$ to be nilpotent (Theorem (\ref{le22})). In dimension
 three only $\ct \hh_1$, where $\hh_1$ denotes the Heisenberg Lie algebra of
 dimension three, posseses a totally real complex structure. 
As application one proves the existence of a generalized complex structure of symplectic type on some kind of  nilpotent Lie algebras $\hh$
and the  existence of lagrangian symplectic structures on $\ct^*\hh$.   

For the coadjoint representation,   we give the general form of totally real 
complex structures $J$ on $\ct^*\hh$, proving the existence when $\hh$ is one
 of the following three dimensional Lie algebras: the Heisenberg Lie algebra, 
 the  Lie algebra of the group of rigid motions of the Minkowski 2-space 
 $\rr_{3,-1}$,  the Lie algebra of the group of rigid motions of the Euclidean 
 2-space $\rr_{3,0}'$ and  the one dimensional trivial central extension of 
 the Lie algebra of the group of affine motions, usually denoted as $\aff(\RR)$.


 In addition to  the Lie algebras obtained in i) the six dimensional tangent Lie 
 algebras admitting complex structures correspond to a Lie algebra $\hh$  
 which is either simple: $\mathfrak{sl}(2)$, $\mathfrak{so}(3)$ or solvable and
 isomorphic to  $\RR \times \aff(\RR)$.
  In   the cotangent case we add   $\mathfrak{sl}(2)$, $\mathfrak{so}(3)$ 
  $\rr_{3,1}$ and $\rr_{3,\eta}'$ for $\eta >0$.

Concerning iii) we study the geometry that derives from the K\"ahler pairs. 
The only Lie algebras carrying such a structure are: the tangent and the 
cotangent of the Heisenberg Lie algebra and   the tangent of 
$\RR\times \aff(\RR)$. The study in the nilpotent case says that there are flat and non flat pseudo K\"ahler metrics, result which extends those in \cite{CFU}. 
Again in this case totally real complex structures provide examples for 
K\"ahler pairs.

\

The second author expresses her gratitude to  the hospitality of the  Departmento de Geometr\'\i a y Topolog\'\i a of the Universidad Complutense de Madrid, where part of this work was written.

\section{Generalities on Complex structures}

An {\it almost complex}  structure on a Lie algebra $\ggo$ is an
endomorphism $J:\ggo \to \ggo$ satisfying $J^2=-{\rm I}$, where
${\rm I}$ is the identity map.

Let $\ggo^{\CC}=\ggo \otimes \CC$  denote the complexification of
$\ggo$ whose elements have the form $v \otimes c$, with $v \in
\ggo,\quad c\in \CC$. An almost complex structure $J$ on $\ggo$
can be extended to a complex linear endomorphism of $\ggo^{\CC}$
that we also denote by $J$, by setting $J(v \otimes c)=Jv \otimes
c$.

As usual, we identify $v\in \ggo$, with $v \otimes 1 \in
\ggo^{\CC}$, and hence any element in $\ggo^{\CC}$ can be written
as $x+iy$ where $x, y \in \ggo$. With this identification, the
eigenspace corresponding to the imaginary eigenvalue $i$ of $J$ is
the subspace $\mm$ of $\ggo^{\CC}$ given by
$$\mm =\{ x - i Jx \, : \, x \in \ggo\}.$$
If we denote by $\sigma$ the conjugation map on $\ggo^{\CC}$, that
is, $\sigma(x+iy)= x-iy$, the eigenspace corresponding to $-i$  is $\sigma \mm$, and we obtain the direct sum of vector spaces
 \begin{equation}\label{des}
 \ggo = \mm \oplus \sigma \mm.
 \end{equation}
 Conversely any decomposition of type (\ref{des}) induces an
 almost  complex structure on $\ggo$. In fact if $x\in \ggo \subseteq \ggo^{\CC}$ and $x$ can be written uniquely as $x=u + v\in \mm \oplus \sigma \mm$ define an endomorphism $J$ by $J x =i u -i v$. Since $\sigma \circ J = J \circ \sigma$, the map $J$ gives rise to an almost complex structure on  $\ggo$. 

The integrability condition of an almost complex structure $J$ is
expressed in terms of the Nijenhuis tensor $N_J$
\begin{equation}\label{NJ}
N_J(x,y)=[Jx,Jy]-[x,y] - J[Jx,y] - J[x,Jy], \qquad\mbox{ for all }
x,y \in \ggo.
\end{equation}

It is straightforward to verify that $N_J(Jx,Jy)=-N_J(x,y)$ for
any $x,y\in \ggo$. Hence, if $\ggo= \uu \oplus J \uu$ is a
 direct sum as vector subspaces, then  $N_J\equiv 0$  if and only if $N_J(u,v)=0=N_J(u,Jv)$ for all $u,v \in \uu$.

An almost complex structure $J$  on $\ggo$ is called {\em
integrable} if $N_J\equiv 0$. In this case  $J$
 is called a {\it complex structure} on $\ggo$. Equivalently, $J$
 is integrable if and only if $\mm$ (and hence $\sigma \mm$) satisfying (\ref{des})
 is a  complex subalgebra of $\ggo^{\CC}$.

Special types of almost complex structures are determined by those
endomorphisms $J :\ggo \to \ggo$ satisfying $J^2 = -{\rm I}$ and one of
the following conditions for any $x, y\in \ggo$:
$$\mbox{c1)}\,\,J[x,y]= [x,Jy] \qquad \qquad
\mbox{c2)}\,\,[Jx,Jy]= [x,y].$$ In any case they are integrable.
Complex structures of type c1) determine a structure of  complex
Lie algebra on $\ggo$,  they are sometimes called {\em bi-invariant}. The subalgebra corresponding to the eigenvalue $\pm i$  is actually an ideal of $\ggo^{\CC}$. 
Structures of type c2) are called {\em abelian}, and the
corresponding eigenspaces for the eigenvalues $\pm i$ are complex
abelian subalgebras of $\ggo^{\CC}$.  It should be remarked that
not any Lie algebra admitting a complex structure can be endowed
with an abelian complex structure, as shows the following example.

\begin{example} \label{exa1} Let $\hh_1$ be the Heisenberg Lie algebra
of dimension three and let $\ct^*\hh_1$ denote the cotangent Lie
algebra of $\hh_1$. This is spanned by the vectors $e_1, e_2, e_3,
e_4, e_5, e_6$ with the non trivial Lie bracket relations
$$[e_1,e_2]=e_3, \quad [e_1, e_6]=-e_5,\quad [e_2, e_6]=e_4.$$
Assume $J:\ct^*\hh_1 \to \ct^*\hh_1$ is an endomorphism satisfying
$[Jx,Jy]=[x,y]$ or, in equivalent form, $[Jx,y]=-[x,Jy]$ for all
$x,y \in \ggo$. If $y$ belongs to the center of $\ct^*\hh_1$, by c2) one
has that $Jy\in \zz(\ct^*\hh_1)$, thus $J$ restricts to the center
$\zz(\ct^*\hh_1)$, and therefore $J$ cannot be almost complex
since the dimension of $\zz(\ct^*\hh_1)$ is odd. To see that
$\ct^*\hh_1$ admits a complex structure see (\ref{coth1}).
\end{example}

 As proved in the previous example, if $\ggo$ carries an abelian complex
structure, then the center of $\ggo$ must be $J$-invariant and therefore even dimensional. Another necessary condition to have
abelian complex structures is that $\ggo$ is 2-step solvable,
which means that the commutator subalgebra $C(\ggo)$ is abelian 
(see \cite{P} for instance). For the sake of completeness we include here a proof. 

Let $\ggo$ be a Lie algebra with a  decomposition as  direct sum of vector subspaces $\ggo=\aa 
\oplus \bb$, where $\aa$ and $\bb$ are abelian subalgebras. Denote $[a,b']=a^*+b''$ and  $[a',b]=a''+b^*$, with $a, a', a'', a^*\in \aa$ and $b,b', b'', b^*\in \bb$. One has
\begin{equation}
\begin{array}{rcl}
[[a,b], [a',b']] & = & [[[a,b], a'], b']-[[[a,b],b'],a']\\
& = & -[[[b,a'],a],b']+ [[[b',a],b],a']\\
& = & [[b^*,a],b'] -[[a^*,b],a']\\
& = & -[[a,b'], b^*]+ [[b,a'], a^*]\\
& = & -[a^*, b^*] -[b^*,a^*]=0, 
\end{array}
\end{equation}

which proves that the Lie bracket on $C(\ggo)$ is trivial. 

Let  $\ggo$ be a Lie algebra and let $J$ be a fixed almost complex
structure on $\ggo$. For any $l\geq 0$ we define the set
$\aa_l(J)$ inductively as: $$\aa_0(J)=\{0\}, \qquad \aa_l(J)=\{ X
\in \ggo,\, /\, [X, \ggo] \subset \aa_{l-1} \mbox{ and } [JX,
\ggo] \subset \aa_{l-1}\} \quad l \geq 1.$$ It is easy to verify
that
$$\aa_0(J) \quad  \subseteq \quad \aa_1(J) \quad  \subseteq \quad \aa_2(J) \quad  \subseteq \quad  \hdots$$

For  a fixed $X\in \aa_{i+1}(J)$ we have that $[X,Y]\in
\aa_{i}(J)\subseteq \aa_{i+1}(J)$ for all $Y\in \ggo$, and clearly
$[J[X,Y],Z]\in \aa_i(J) \subset \aa_{i+1}(J)$ for all $Y,Z\in
\ggo$. Therefore $\aa_i(J)$ is a $J$-invariant ideal of $\ggo$ for
any $i\geq 0$.

The almost complex structure $J$ is called {\em nilpotent} if there
exists a $t$ such that $\aa_t(J) =\ggo$. This
implies that $\ggo$ must be nilpotent. For a nilpotent almost
complex structure $J$ on a $s$-step nilpotent  Lie
algebra of even dimension we shall say that it is $r$-step nilpotent if $r$ is the
first nonnegative integer such that $\aa_r(J)=\ggo$; this
satisfies the inequality $s\leq r \leq n$ \cite{CFGU2}. In the
following sections examples are given that show that these bounds
are actually reached (see (\ref{coth1}) and (\ref{tah1})). Notice
that if $J$ is a nilpotent almost complex structure  on a
nilpotent Lie algebra $\ggo$, then any term of the ascending
series of $\ggo$ admits  a two dimensional $J$-invariant subspace.
Clearly, if $J$ is integrable, the condition of  being nilpotent
is stronger than asking the corresponding $\mm$ for $J$ to be
nilpotent.

Notice $a_1(J)\subseteq \zz(\ggo)$. If $J$ is abelian then the equality holds
$a_1(J)=\zz(\ggo)$.

\begin{example} The canonical complex structure of a nilpotent complex Lie algebra is nilpotent (see (\ref{cn})).
\end{example}

An equivalence relation is defined among Lie algebras with complex
structures.  Lie algebras with complex structures $(\ggo_1,J_1)$
and $(\ggo_2,J_2)$ are called {\em holomorphically equivalent} if there exists an
isomorphism of Lie algebras $\alpha:\ggo_1 \to \ggo_2$ such that
$J_2 \circ \alpha = \alpha \circ J_1$. In particular when
$\ggo_1=\ggo_2$ we simply  say that $J_1$ and $J_2$ are {\em equivalent} and a classification of complex structures can be
done. 

\begin{lem} \label{le} Let $\ggo$ be an even dimensional  real Lie algebra.

i) The class of an abelian complex structure, if non-empty,
consists only of abelian complex structures.

ii) Let  $J, J'$ be complex structures on  $\ggo$ such that
$J'=\sigma J \sigma^{-1}$. Then $\sigma \aa_l(J)=\aa_l(J')$ for
any $l\geq 0$.

In particular the class of a nilpotent complex structure on a
given nilpotent Lie algebra consists only of nilpotent complex
structures, all of them being  nilpotent of the same type.

iii) The class of a bi-invariant complex structure has only bi-invariant complex structures.
\end{lem}
\begin{proof} i) Let $J$ an abelian complex structure on $\ggo$.
If $J'=\sigma J \sigma^{-1}$, the result follows using that
$\sigma$ is an automorphism and $J$ is abelian.

ii) For $i=0$ it is clear that $\sigma \aa_0(J)=\aa_0(J')$. Assume that $\sigma \aa_l(J) =\aa_l(J')$ for all $l\leq i$. Let $X\in \aa_{i+1}(J)$. It follows at once that the following identity is satisfied
$$[\sigma X, \ggo]=[\sigma X, \sigma \ggo]=\sigma [X, \ggo] \in \sigma \aa_i(J) =\aa_i(J').$$
On the other hand,
$$[J'\sigma X, \ggo]=[\sigma J X, \ggo]=\sigma [J X, \ggo]\in \sigma \aa_i(J) =\aa_i(J').$$
Therefore $\sigma \aa_{i+1}(J) \subset \aa_{i+1}(J')$. A similar
argument, interchanging the roles of $J$ and $J'$, proves that
$\sigma \aa_i(J)=\aa_i(J')$.

iii) follows by direct application of the definitions.
\end{proof}

\begin{remark} From the definitions above it is not immediately
clear which is the relationship  between nilpotent complex
structures and complex structures whose corresponding
$i$-eigenspace $\mm$ is nilpotent. In (\ref{tah1}) we show that
the tangent Lie algebra of the Heisenberg Lie algebra $\hh_1$
carries only 2-step nilpotent complex structures, some of them
being abelian, and others having $\pm i$-eigenspaces which are
2-step nilpotent subalgebras.
\end{remark}

\vskip 3pt

\section{Totally real complex structures on tangent and cotangent Lie algebras}

The aim of this section is the study of  totally real complex structures on
tangent and cotangent  Lie algebras, that is complex structures $J$ on $\ct_{\pi} \hh$ such that $J\hh=V$. 

We briefly recall  the construction. Let $\hh$ denote a real Lie algebra and
let $(\pi,V)$  be a finite dimensional representation of $\hh$. By
endowing $V$ with the trivial Lie bracket, consider the semidirect product of $\hh$ and $V$ relative to $\pi$,  $\ct_{\pi} \hh:= \hh \ltimes_{\pi} V$,  where the Lie bracket is:
$$[(x,v), (x', v')]=([x,x'], \pi(x) v' - \pi(x') v)\qquad x,x'\in \hh, v,v'\in V.$$

In this work we focus on the  the adjoint and the coadjoint representations.
In both cases    $V$ is a real vector space with the same dimension as that of $\hh$. The adjoint representation $\ad: \hh \to \mathfrak{gl}(\hh)$ is given by
$\ad(x) y = [x,y]$,  and it defines  the {\em tangent} Lie algebra that we denote with  $\ct \hh$.

 For the coadjoint representation $\ad^*:\hh \to \mathfrak{gl}(\hh^*)$, that is
 $V=\hh^*$, and 
\begin{equation}\label{coad}
\ad(x)^* \varphi (y) = - \varphi \circ \ad(x) y \qquad \mbox{ for } x, y
\in \hh, \, \varphi \in \hh^*; \end{equation}
we call  the resulting
 Lie algebra {\em cotangent Lie algebra} and we denote it as $\ct^*\hh$.

A   question concerning complex structures when we look at the algebraic structure of the   Lie algebra  $\ct_{\pi} \hh=\hh \ltimes_{\pi} V$  is whether there exists  an almost complex structure $J$ such that $J \hh=V$. Such a $J$ induces a linear isomorphism $j:\hh \to V$, and conversely any such $j:\hh \to V$ determines an almost complex structure on $\ct_{\pi} \hh$ such that $J\hh=V$, by means of
\begin{equation}\label{j}
J(x,v)=(-j^{-1} v, j x) \qquad x\in \hh, v\in V, \,j: \hh \to V.
\end{equation}

It follows that both $\hh$ and $V$ are totally real with respect to $J$.
 We adopt the next terminology, following  \cite{CLP}.  

  \begin{defn} Let $\ct_{\pi} \hh:=\hh\ltimes_{\pi} V$ be  the semidirect product of a 
  Lie algebra $\hh$ with the real vector space $V$ such that $\dim V=\dim \hh$  and let
   $J$ denote an (almost) complex structure on $\ct_{\pi}\hh$.  If   $J\hh=V$ we say 
   that $J$ is  a {\em totally real} (almost) complex structure on $\ggo$.
   \end{defn}
   
   Assume $(\pi, V)$ denotes a finite dimensional representation of $\hh$ and let $J$ 
   be a totally real  almost complex structure on $\ct_{\pi} \hh= \hh \ltimes_{\pi} V$ like in (\ref{j}). In this case, the integrability condition for $J$ reduces to 
\begin{equation}\label{e12}
0 = [x,y] - j^{-1} \pi(x) j y + j^{-1} \pi(y) j x\qquad \mbox{ for all } x, y\in \hh.
\end{equation}

Fix  a Lie algebra $\hh$, recall that the representations $(V, \pi)$ and $(V',\pi')$ are called {\em equivalent} if there is a linear isomorphism $T:V \to V'$ such that $T^{-1} \pi'(x) T = \pi(x)$ for all $x\in \hh$. 

Actually for any $\psi\in Aut(\hh)$, the map $\varphi:\ct_{\pi} \hh \to \ct_{\pi'} \hh$ given by $\varphi=\psi+T$ is a Lie algebra isomorphism. In fact for all $x,y\in \hh$, $u,v\in V$
$$\begin{array}{rcl}
\varphi[x+u, y+v] & = & \varphi([x,y]+\pi(x) v - \pi(y)u)\\
 & = & \psi [x,y]+ T\pi(x) v - T\pi(y)u\\
 & = & [\psi x, \psi y] + \pi'(x) T v - \pi'(y) Tu\\
 & = & [\varphi(x+u), \varphi(y+v)]
\end{array}
$$
And therefore if $J$ denote a complex structure on $\ct_{\pi}\hh$ then $J':=\varphi \circ J \circ \varphi^{-1}$ denotes a complex structure on $\ct_{\pi'} \hh$ making of $(\ct_{\pi} \hh, J)$ and $(\ct_{\pi'} \hh, J')$ a pair of  holomorphically equivalent Lie algebras.

In particular if  $J$ is a totally real complex structure on $\ct_{\pi} \hh$, then  $(\ct_{\pi}\hh, J)$ is holomorphically equivalent to $(\ct_{\pi'}\hh, \tilde{J})$ where $\tilde{J}_{|{_\hh}}:\hh \to V'$ is $\tilde{J}=T \circ J_{|_{\hh}}$ and extended as in (\ref{j}). The proof of the following result follows by using these relations and the integrability condition (\ref{e12}).

\begin{prop} \label{ceq} Let $(V, \pi)$ and $(V', \pi')$ be equivalent representations of a Lie algebra $\hh$ such that $\dim V = \dim V'=\dim \hh$. Complex structures on $\ct_{\pi}\hh$ are in one to one correspondence with complex structures on $\ct_{\pi'} \hh$.

In particular, totally real complex structures on $\ct_{\pi} \hh$ are homomorphically equivalent to totally real complex structures on $\ct_{\pi'} \hh$. 
\end{prop}

 A  first consequence of (\ref{e12}) concerns abelian complex structures.

\begin{cor} \label{ca} Let $\hh$ be a Lie algebra and let $V$ denote the underlying vector space of $\hh$. Let $\ct_{\pi} \hh:=\hh \ltimes_{\pi} V$ denote the semidirect product and let $J$ be an abelian totally real complex structure on $\ct_{\pi} \hh$. Then $\hh$ is abelian and $\pi$ and $J$ are related by $\pi(x) Jy = \pi(y) Jx$ for all $x,y\in \hh$.
\end{cor}

\begin{remark} The converse of the previous corollary is true. Let $\hh$ denote an abelian Lie algebra and let $\pi$ be a representation of $\hh$ into $\hh$. Then if $j:\hh\to \hh$ is a non singular map such that $\pi(x)jy =\pi(y) jx$ for all $x,y\in \hh$, then the almost complex structure on $\hh \ltimes_{\pi} \hh$ given as in (\ref{j}) is integrable and totally real with respect to $\hh$. See the final section for an explicit example.
\end{remark}

\begin{example} Consider $\RR^n$ with the canonical basis $\{e_1, e_2, \hdots, e_n\}$ and let $A$ be a non singular $n\times n$ real matrix. Let $C_A$ denote the centralizer of $A$ in $gl(n,\RR)$ that is, the set of $n\times n$ matrices $B$ such that $BA=AB$. Let $B_1, B_2, \hdots ,B_n$ be $n$ matrices in $C_A$ such that they are pairwise in involution, $B_i B_j = B_j B_i$ for all $i,j$. Take $\pi$ the representation of $\RR^n$ which extends linearly the mapping $e_i \to B_i$ (notice that this could be trivial depending on $A$). The map $j$ represented by $A$ amounts to a totally real abelian complex structure on $\ct_{\pi} \RR^n$. 
\end{example}

When $\pi$ is the adjoint representation, the solutions of (\ref{e12}) have an algebraic interpretation as it will be  seen next. Recall first that a {\em derivation} of a Lie algebra $\hh$ is a linear map $d:\hh \to \hh$ such that
$$d[x,y]=[dx, y] + [x,dy]\qquad \qquad \mbox{ for all } x,y \in \hh.$$

Jacobson proved  that if a Lie algebra $\hh$ admits a non singular derivation then it must be nilpotent  \cite{Ja}. 

\begin{thm} \label{le22} Let $\ct \hh$ denote the tangent Lie algebra of $\hh$. The set 
of  totally real complex structures on $\ct \hh$   is in one to one correspondence with the set of non singular derivations of $\hh$. 
If one (and therefore both) of these sets is non empty, then $\hh$ is nilpotent.
\end{thm}
\begin{proof} Let $\ad$ denote  the adjoint representation of $\hh$. The integrability condition   (\ref{e12}) becomes
$$ 0 = j[x,y] - \ad(x) jy + \ad(y) jx \qquad \mbox{ for all } x,y \in \hh.$$
This shows, via identifications,  that the complex structure $J$ determined by $j$ (\ref{j}) corresponds to a  non singular derivation of $\hh$. The proof is completed after the application of the result of Jacobson.
\end{proof}

\begin{example}\label{derh1}  Let $\hh_1$ denote the Heisenberg Lie algebra of dimension three (see \ref{lie3}). By Lemma (\ref{le22}) if $J$ is a totally real complex structure on  $\ct \hh_1$, then it corresponds to a non singular derivation of $\hh_1$. Any non singular derivation of $\hh_1$ has a matrix representation  in the basis of (\ref{lie3}) given by:
$$ \left( \begin{matrix} A  & 0 \\ * & tr(A)  \end{matrix}
 \right),\qquad \mbox{ with } A\in GL(2,\RR)\quad \mbox{ and } tr(A)\neq 0$$
 where  $tr$ denotes the trace of the matrix $A$. 
 
 More generally any   non singular derivation $d$ of the Heisenberg Lie algebra $\hh_n$ 
 of dimension $2n+1$, has a matrix representation as above with $A\in GL(2n, \RR)$ and $tr(A)\neq 0$. This induces a totally real complex structure $J$ on the tangent Lie algebra $\ct \hh_n$.
\end{example}

\subsection{Totally real complex structures on six dimensional cotangent Lie algebras}
We now proceed to analyze the existence of totally real complex structures  on six dimensional cotangent Lie algebras $\ct^*\hh$. To this extend, recall  the  classification of three dimensional Lie algebras as given e.g. in \cite{GOV} or \cite{Mi}.

\begin{thm}\label{lie3} Let $\hh$ be a real  Lie algebra of dimension three spanned by $e_1,e_2,e_3$.
Then it is isomorphic to one and only one in the following list:
\begin{equation}
\begin{array}{rll}
\hh_1 & [e_1,e_2]=e_3 \\
\rr_{3} & [e_1,e_2]=e_2,\, [e_1,e_3]= e_2 + e_3 \\
\rr_{3,\lambda} & [e_1,e_2]=e_2,\, [e_1,e_3]= \lambda e_3 &  |\lambda| \leq 1 \\
\rr_{3,\eta}' & [e_1,e_2]=\eta e_2- e_3,\, [e_1,e_3]= e_2 + \eta e_3 & \eta \geq 0\\
\mathfrak{sl}(2) & [e_1,e_2]=e_3,\, [e_3,e_1]= 2e_1,\, [e_3,e_2]= -2e_2\\
\mathfrak{so}(3) & [e_1,e_2]=e_3,\, [e_3,e_1]= e_2,\, [e_3,e_2]= -e_1
\end{array}
\end{equation}
\end{thm}

A Lie algebra $\ggo$ which satisfies $tr(\ad(x))=0$ for all $x\in
\ggo$ is called  unimodular. Among the Lie algebras above, the
unimodular ones are: $\hh_1$, $\rr_{3,-1}$ and $\rr_{3,0}'$. The
Lie algebra $\hh_1$ is known as the Heisenberg Lie algebra, while
$\rr_{3,-1}$ is the Lie algebra of the group of rigid motions of
the Minkowski 2-space and $\rr_{3,0}'$ corresponds to the Lie
algebra of the group of rigid motions of the Euclidean 2-space.
The Lie algebra $\rr_{3,0}$ denotes the central extension of the
Lie algebra of the group of affine motions in $\RR$, usually
denoted as $\aff(\RR)$.

\begin{prop}\label{trct} Let $\ct^*\hh=\hh \ltimes \hh^*$ be a cotangent
Lie algebra of a three dimensional Lie algebra $\hh$. Then totally real complex structures  on $\ct^{\ast}\hh$ exist whenever $\hh$ is either
unimodular  or isomorphic to  $\RR \times \aff(\RR)$. In those
cases the map $j:\hh \to \hh^*$ admits a matrix representation as
follows
$$\ct^* \hh_1\quad \left( \begin{matrix} a_{41} & a_{42} & a_{43}
 \\ a_{51} & a_{52} & a_{53} \\ -a_{43} & -a_{53} & 0 \end{matrix}
 \right); \qquad \ct^*\rr_{3,-1} \quad \left( \begin{matrix} a_{41}
 & a_{42} & a_{43} \\ -a_{42} & 0 & a_{53} \\ -a_{43} & -a_{53} & 0 \end{matrix} \right);$$
$$\ct^*\rr_{3,0} \quad \left( \begin{matrix} a_{41} & a_{42} & a_{43}
\\ -a_{42} & 0 & 0 \\ a_{61} & 0 & a_{63} \end{matrix} \right); \qquad
\ct^*\rr_{3,0}' \quad \left( \begin{matrix} a_{41} & a_{42} &
a_{43} \\ -a_{42} & 0 & a_{53} \\ -a_{43} & -a_{53} & 0
\end{matrix} \right);$$ where the matrix should be non singular.
\end{prop}
\begin{proof} The proof follows by direct computation of  (\ref{e12}) taking  $\pi$ as
 the coadjoint representation. In the cases not listed above, the maps $j$ solving (\ref{e12})
are singular, hence they cannot induce a complex structure on
$\ct^* \hh$.
\end{proof}

\begin{example} Let $\sg$ denote a semisimple  Lie algebra. Since the Killing form is non degenerate this induces an ad-invariant metric on $\sg$. Therefore the adjoint and coadjoint representation are equivalent. 
By (\ref{ceq}) the existence of totally real complex structures on $\ct^*\sg$ reduces to the existence of totally real complex structures on $\ct \sg$ and this cannot admit a totally real complex structure by (\ref{le22}). Observe  that in dimension three the simple Lie algebras are $\mathfrak{sl}(2)$ and $\mathfrak{so}(3)$ (see (\ref{lie3})).
\end{example}

\section{complex structures  on tangent and cotangent Lie algebras of dimension six}

Examples of six dimensional real Lie algebras  with
complex structures arise from three dimensional complex Lie
algebras. In fact let $\tilde{\ggo}$ denote a three dimensional
complex Lie algebra, then the underlying
real Lie algebra $\ggo:=\tilde{\ggo}_{\RR}$,  is naturally equipped with a bi-invariant
complex structure  induced by the multiplication by $i$
on $\tilde{\ggo}$. In this way this complex structure on $\ggo$
is bi-invariant.

\begin{example} \label{cn} Let $\ggo$ denote a six dimensional two-step nilpotent Lie algebra equipped with bi-invariant complex structure $J$. We shall see that $\ggo$ is isomorphic to the real Lie algebra underlying $\hh_1\otimes \CC$, the complexification of the Heisenberg Lie algebra of dimension three. 

 Since $J \circ \ad(x) = \ad(x) \circ J$ for any $x\in \ggo$, we obtain the inclusions   $J\zz(\ggo)\subseteq \zz(\ggo)$ and $J C(\ggo)\subset C(\ggo)$. Furthermore, there is a central ideal  $\vv \subseteq \zz(\ggo)$ such that
 $$\zz(\ggo)=C(\ggo) \times \vv \qquad \mbox{ direct sum of $J$-invariant Lie algebras.}$$ 
 
 A way to produce such a $\vv$ is the following. Take $\la \,, \, \ra$ an inner product on $\zz(\ggo)$ which is Hermitian for $J$ and let $\vv =C(\ggo)^{\perp}$. 

Thus $\ggo=\vv \times \nn$ where $\nn$ is a two-step nilpotent Lie algebra such that $C(\ggo)=C(\nn)=\zz(\nn)$ and $\nn$  is equipped with a bi-invariant complex structure $J$, induced from that one on $\ggo$. Now if $\vv$ is non empty then it may be two or four dimensional. If it is two dimensional, then $\nn$ is four dimensional and is equipped with a bi-invariant complex structure, therefore it must be abelian (in dimension four a solvable Lie algebra endowed with a bi-invariant complex structure is either abelian or isomorphic to $\aff(\CC)$, see \cite{O2} for instance). A similar reasoning applies when $\vv$ has dimension four, and therefore $\zz(\ggo)=C(\ggo)$.

Now let $z\in \zz(\ggo)= C(\ggo)$. Then there exist $x, y\in \ggo$ such that $[x,y]=z$. The set $\{x, y, z\}$ is linearly independent and the set $\{x,y,z, Jx, Jy, Jz\}$ is a basis of $\ggo$. Due to the bi-invariance condition on $J$ one has the following Lie bracket relations
$$[x,y]=z\quad [Jx,y]=Jz\quad [x,Jy]=Jz\quad [Jx,Jy]=-z$$
moreover $[x,Jx]= 0 =[y,Jy]$, and therefore the Lie algebra $\ggo$ is isomorphic to the real Lie algebra of dimension six underlying $\hh_1 \otimes \CC$.
\end{example}

In the previous section we gave  examples of  complex
structures in tangent and cotangent Lie algebras of dimension six. Now we shall study  the
existence problem of complex structures on any  tangent or cotangent Lie 
algebra corresponding to a three dimensional  real Lie algebra as in 
(\ref{lie3}). Some considerations about abelian and
nilpotent complex structures are also be given.

\vskip 3pt

{\bf The simple case}. Among the Lie algebras listed in (\ref{lie3}) the simple ones are $\mathfrak{sl}(2)$ and $\mathfrak{so}(3)$. Since the Killing form is non degenerate in both cases, the adjoint and the coadjoint representations are equivalent. After (\ref{ceq}) for a semisimple Lie algebra $\mathfrak{s}$, the existence of complex structures on $\ct \mathfrak{s}$  determines  one on $\ct^*\mathfrak{s}$ and viceversa. Recall that complex structures on compact semisimple and more generally on reductive Lie algebras were extended studied (see for instance \cite{Sm}\cite{Wa} \cite{Sa1} \cite{Sa2} \cite{SD}). We perform below a construction of a complex structure on $\ct \mathfrak{sl}(2)$ and $\ct \mathfrak{so}(3)$.  
Explicitly the Lie brackets  are given by:
$$\begin{array}{llll}
\ct \mathfrak{sl}(2) & [e_1,e_2]=e_3 & [e_3,e_1]= 2e_1 & [e_3,e_2]= -2e_2\\
 & [e_1,e_5]=e_6 & [e_1,e_6]= -2e_4 & [e_2,e_4]= -e_6\\
 & [e_2,e_6]=2e_5 & [e_3,e_4]= 2e_4 & [e_3,e_5]= -2e_5\\
\ct \mathfrak{so}(3) & [e_1,e_2]=e_3 & [e_3,e_1]= e_2 & [e_3,e_2]= -e_1\\
& [e_1,e_5]=-e_6 & [e_1,e_6]= e_5 & [e_2,e_4]= -e_6\\
 & [e_2,e_6]=e_4 & [e_3,e_4]= -e_5 & [e_3,e_5]= e_4
 \end{array}
 $$
 In any case  an almost complex structure $J$ can be defined by
 $$J e_3= e_6\qquad Je_2=e_1 \qquad Je_4= e_5.$$
 By calculating $N_J$ one verifies that $J$ is integrable. Hence 
 
 \vskip 4pt
 
 {\it The tangent Lie algebras $\ct \mathfrak{so}(3)$ and $\ct \mathfrak{sl}(2)$ (and therefore
 $\ct^*\mathfrak{so}(3)$ and $\ct^* \mathfrak{sl}(2)$) carry  complex structures}

\

{\bf The solvable case.} Suppose that $\ggo$ is a  six dimensional tangent $\ct \hh$ or cotangent Lie
algebra $\ct^* \hh$ being $\hh$ a solvable real Lie algebra of dimension three. It admits a complex structure if and only if $\ggo^{\CC}$
decomposes as a direct sum of vector subspaces $\ggo^{\CC}=\mm
\oplus \sigma \mm$, where $\mm$ (resp. $\sigma \mm$) is a complex
subalgebra. Without lost of generality assume that $\mm$ is spanned by vectors $U, V, W$ as follows:
\begin{equation}\label{uv} U= e_1 + a_2 e_2 + a_3 e_3 + a_4 e_4 + a_5 e_5 + a_6 e_6
\quad V= b_2 e_2 + b_3 e_3 + b_4 e_4 + b_5 e_5 + b_6 e_6,
\end{equation}
$$ W=  c_2 e_2 + c_3 e_3 + c_4 e_4 + c_5 e_5 + c_6 e_6, \quad a_i, b_j, c_k\in \CC, \forall i, j, k=2, \hdots 6.$$
Let $\aa:=span\{V,W\}$. We claim that $\aa$ is a ideal in $\mm$. In fact, 
according to the Lie brackets in $\ggo$ (see (\ref{ta3}) and (\ref{cot3}) below),  
one verifies that  $U\notin  C(\ggo)$, hence for any $x,y\in \mm$, 
$[x,y]\in C(\mm) \subseteq \aa$. Thus $\mm =\CC U \ltimes \aa$, being  
$\aa$  a ideal of $\mm$ of dimension two and therefore isomorphic either to 
i) $\CC^2$ or to ii) $\aff(\RR)$,  the two dimensional complex Lie
algebra spanned by vector $X, Y$ with $[X,Y]=Y$. 
We may assume in the last situation that $V,W$ satisfy the Lie bracket relation 
$[V,W]=W$.

In case $\mm=\CC U \ltimes \CC^2$, the action of $U$ on $\aa$ admits a basis whose matrix is one of the following ones
\begin{equation}\label{typ1}
\begin{array}{ll}{\mbox{(1)}\quad \left( \begin{matrix}
\nu & 0\\ 0 & \mu \end{matrix} \right),\qquad   \nu,\, \mu \in
\CC};& {\qquad \qquad \qquad \mbox{(2)} \quad \left( \begin{matrix}
\nu & 1\\ 0 & \nu \end{matrix} \right),\qquad \nu \in \CC}.\\
\end{array}
\end{equation}

In case $\mm = \CC U \ltimes \aff(\RR)$ the action of $U$ on $\aa$ is a derivation 
of $\aff(\RR)$ thus over the basis $\{V,W\}$ we have a matrix 
\begin{equation}\label{typ2}
\left( \begin{matrix}
0 & 0\\ a & b \end{matrix} \right)
\qquad \qquad a, b\in \CC
\end{equation}

By making use of this we shall derive the existence or non existence of complex structures on any  tangent or cotangent Lie algebra corresponding to a three dimensional solvable real Lie algebra.

\subsection{Complex structures on six dimensional tangent Lie algebras}
If $H$ denotes a Lie group, its tangent bundle $T H$ is
identified with $H\times \hh$, which inherits a natural Lie group
structure as the semidirect product under the adjoint
representation. Its Lie algebra,  the tangent Lie algebra
$\ct \hh$, is the semidirect product via the adjoint
representation $\hh \ltimes_{\ad} V$, where $V$ is the
underying vector space to $\hh$ equipped with the trivial Lie
bracket.

\begin{prop} \label{ta3} Let $\hh$ be a solvable real  Lie algebra of
dimension three  and let $\ct \hh$ denote the tangent Lie
algebra spanned by $e_1, e_2, e_3, e_4, e_5, e_6$.  Then the non zero Lie brackets are presented in the following list:
$$
\begin{array}{ll}
\ct \hh_1: & [e_1,e_2]=e_3, \, [e_1, e_5]=e_6, [e_2, e_4]= -e_6\\
\ct \rr_{3}: & [e_1,e_2]=e_2,\, [e_1,e_3]= e_2 + e_3,\, \\
& [e_1, e_5]= e_5,\, [e_1, e_6]=e_5 + e_6,\, [e_2, e_4]=-e_5,\, [e_3, e_4]=-e_5-e_6 \\
\ct \rr_{3,\lambda}: & [e_1,e_2]=e_2,\, [e_1,e_3]= \lambda e_3  \\
  |\lambda| \leq 1& [e_1, e_5]= e_5,\, [e_1, e_6]=\lambda e_6,\, [e_2, e_4]=-e_5,
\, [e_3,e_4]=-\lambda e_6\\ \ct \rr_{3,\eta}': & [e_1,e_2]=\eta
e_2- e_3,\, [e_1,e_3]= e_2 + \eta e_3,
 [e_1, e_5]= \eta e_5-e_6,\,  \\
\eta \geq 0 &  [e_1, e_6]= e_5+\eta e_6,\, [e_2,e_4]=-\eta e_5+e_6,\,   [e_3, e_4]= -e_5-\eta  e_6\\
\end{array}
$$
\end{prop}

\vskip 3pt

\begin{thm} \label{teot} Let $\hh$ denote a three dimensional Lie algebra, then $\ct \hh$ admits a complex structure if and only if $\hh$ is either isomorphic to  $\hh_1$ or  $\RR \times \aff(\RR)$.
\end{thm}

The proof can be derived from the next paragraphs.

\begin{lem} If $\mm$ is a complex subalgebra of $\ct \hh$ being $\hh$ a three
dimensional solvable real Lie algebra such that $\ct \hh^{\CC}=\mm \oplus \sigma \mm$ then 
$\mm \simeq \CC \ltimes \CC^2$.
\end{lem}
\begin{proof} According to the previous  paragraphs it should hold $\mm\simeq \CC
\ltimes \CC^2$ or $\mm \simeq \CC \ltimes \aff(\RR)$. We shall prove that the last
situation is not possible. In fact, from the Lie brackets above (\ref{ta3}) we see
that $[V,W]\in span\{e_5, e_6\}$ so that $c_2=0=c_3=c_4$. But by computing one has
$[V,W]=0$ implying $W=0$ and therefore no complex structure can be derived from this
situation. 
\end{proof} 

With the previous Lemma it follows to analyze next the existence of complex structures
attached to complex Lie subalgebras $\mm$ such that $\mm \simeq \CC \ltimes \CC^ 2$.

\vskip 3pt

We already know that the tangent Lie algebra of the three dimensional Heisenberg Lie algebra $\ct \hh_1$, admits complex structures, moreover totally real ones. 
 Recall that any totally real complex structure on $\ct \hh_1$ corresponds to a non singular derivation of $\hh_1$ (\ref{derh1}). No one of these complex structures is abelian. However $\ct \hh_1$ can be equipped with abelian complex structures as we show below.

Let $\mm$ be a complex subalgebra of $\ct \hh_1$ spanned by 
vectors $U,V,W$ as in (\ref{uv}).
The subspace $\aa=span\{V,W\}$ is a ideal of $\mm$ and $\mm = \CC U \ltimes \aa$. Since $\ct \hh_1$ is nilpotent, $\aa$ is abelian and the action of $U$ on $\aa$ is of type (\ref{typ1}) and moreover case 1) holds for $\mu=\nu=0$ while case 2) holds for $\nu =0$.
Case 1) gives rise to abelian complex structures, while case 2) corresponds to non abelian ones. 

Computing the Lie brackets
$[V,W]$, $[U,V]$ and $[U,W]$, and imposing these brackets to be
zero, we get 
$$U = e_1 + a_2 e_2 + a_3 e_3 +a_4 e_4 + a_5 e_5 + a_6 e_6,$$
$$V=   b_3 e_3 + b_4 e_4 -a_2 b_4  e_5 + b_6 e_6,\qquad W= c_3 e_3+ c_4 e_4 - a_2 c_4 e_5 + c_6 e_6.$$ 
If the set $\{U,V,W, \sigma U, \sigma V, \sigma W\}$ spans a basis of $(\ct
\hh_1)^{\CC}$, the tangent algebra $\ct \hh_1$ carries an abelian
complex structure $J$. For instance the linear homomorphism $J$  given by 
\begin{equation}\label{h1abe}
J e_1 = e_2 \qquad J e_6 = e_3 \qquad J e_4 = e_5, \end{equation}
and such that $J^2=-{\rm I}$ defines a abelian complex structure on $\ct \hh_1$. After \cite{Mg} there is only one class among abelian complex structures.

Any abelian complex structure is 2-step nilpotent. 
In fact, since $J$ is abelian  $\aa_1(J)=\zz(\ct \hh_1)$
and clearly the condition $C(\ct \hh_1) = \zz(\ct \hh_1)$ shows
that $\aa_2(J)=\ct \hh_1$.

On the other hand  the following set of vectors on $\ct \hh_1^{\CC}$  is a basis of the complex subalgebra $\mm$ corresponding to a totally real complex structure on $\ct \hh_1$:
$$e_1 - i (a e_4 +  b e_5 + e e_6); \quad e_2 -i (c e_4 + d e_5+ f e_6), \quad e_3 - i (a+d) e_6 $$
with $a, b, c, d, e, f\in\RR$, $a+d\neq 0$ and $ad-bc \neq 0$. They induce non 
abelian complex structures, and furthermore  there are more non abelian 
complex structures than the totally real ones. Let  $\mm$ be a complex 
subalgebra of $(\ct \hh_1)^{\CC}$
spanned by $U,V,W$ as in (\ref{uv}). Requiring that
$[U,V]=0=[V,W]$  and  $[U,W]=V$ we deduce that any complex subalgebra $\mm$ of $(\ct
\hh_1)^{\CC}$ spanned by  
$$U = e_1 + a_2 e_2 + a_3 e_3 +a_4 e_4 + a_5 e_5 + a_6 e_6,$$
$$V= c_2 e_3+ (c_5 - a_2 c_4 + a_4 c_2)e_6 \qquad W=  c_2 e_2 + c_3 e_3 + c_4 e_4 + c_5  e_5 + c_6 e_6,\qquad $$ 
and such that $U,V, W, \sigma U, \sigma V, \sigma W$ is a basis of $(\ct
\hh_1)^{\CC}$, induces a non abelian complex structure on $\ct \hh_1$. 

The class of non abelian  complex structures $J$ is 2-step nilpotent. Actually the
vector $X:= W +\sigma W$ belongs to the center of $\ct \hh_1$ and
also $JX\in \zz(\ct \hh_1)$. Since $\aa_1(J)=span\{X,JX\}=\zz(\ct
\hh_1)$ and $C(\ct \hh_1)=\zz(\ct \hh_1)$, we conclude that
$\aa_2(J)=\ct \hh_1$. 

After \cite{Mg} in the set of non abelian complex structures, one has the following non equivalent complex structures (the extension is such that $J^2=-{\rm I}$):
\begin{equation}\label{h1tr}
J_s e_1= e_4\qquad J_s e_2 = -s e_4+e_5 \qquad J_s e_3= 2e_6\qquad s=0,1,
\end{equation}
which are totally real, and next
\begin{equation}\label{h1ntr}
J_{\nu} e_1= e_2 + (1-\nu) e_4+\frac{1-\nu}\nu e_5\quad J e_2=-\nu e_1 + (1-\nu) e_4\quad J_\nu e_3=e_6\quad \nu \in \RR-\{0\}
\end{equation}
which are neither abelian nor totally real. 

\begin{prop} \label{tah1} The tangent Lie algebra $\ct \hh_1$ admits  abelian and nonabelian complex structures, which are in every case 2-step nilpotent.
\end{prop}

\begin{remark} The Lie algebra $\ct \hh_1$ is isomorphic to $\mathcal G_{6,1}$ in \cite{Mg} and to $\mathfrak h_4$ in \cite{CFU}.
\end{remark} 

Next we show a family of tangent Lie algebras which cannot be equipped with a complex structure.

\begin{prop} \label{noct} The Lie algebras  $\ct \rr_3$ and $\ct \rr_{3, \eta}'$ $\eta \geq 0$ do not admit  complex structures.
\end{prop}
\begin{proof}
We shall give the complete proof for the tangent Lie algebra $\ct \rr_3$. The 
proof for $\ct \rr_{3, \eta}'$ can be achieved with a similar reasonning. 

Assume  that $\mm$ is a complex subalgebra of $\ct \rr_3^{\CC}$ corresponding to  a complex structure of $\ct \rr_3$.  
Let $U, V, W\in \mm$ be linearly independent vectors  as in (\ref{uv}) such that $\mm= \CC U \ltimes \aa$ with 
$\aa=span\{V,W\}$. 

Let $\mm$ be a complex subalgebra of $\ct \rr_3^{\CC}$ spanned by vectors $U,V, W$ as in (\ref{uv}) and assume that 
$\aa=span\{V,W\}$ is abelian.
The Lie bracket relations on $\mm$ are:
$$\begin{array}{rcl}
[V,W] & = & (b_4 c_2 - b_2 c_4 + b_4 c_3 - b_3 c_4) e_5 + (b_4 c_3 - b_3 c_4) e_6, \\
\, [U,V] & = & (b_2 + b_3)  e_2 + b_3 e_3 + (b_5 + b_6+a_4 b_2 - a_2 b_4 + a_4 b_3 - a_3 b_4) e_5 + \\
 & &  + (b_6+a_4 b_3 - a_3 b_4)  e_6,\\
 \, [U,W] & = & (c_2 + c_3)  e_2 + c_3 e_3 + (c_5 + c_6+a_4 c_2 - a_2 c_4 + a_4 c_3 - a_3 c_4) e_5 + \\
 & &  + (c_6+a_4 c_3 - a_3 c_4)  e_6.
 \end{array}
 $$

Assume $[V,W]=0$, $[U,V]=\nu V$ and $[U,W]= \mu W$. 

We see  that $\nu b_4=0$ implies either $\nu=0$ or $b_4=0$. 

If $\nu=0$, from $[U,V]=0$ one has $b_2=0=b_3$ and $V$ has the form $V=b_4 (e_4 + a_2 e_5 + a_3 e_6)$.

 We also have $\mu c_4 = 0$. If $\mu=0$ then $c_2 =0=c_3$ and so $U,V, \sigma V \sigma W\in span\{e_4, e_5, e_6\}$;
  therefore this set is not linearly independent. Hence $\mu \neq 0$  and $c_4=0$. But from $[V,W]=0$ it must hold 
  $b_4 c_3=0=b_4 c_2$ and  since $b_4\neq 0$ one has $c_2 = 0=c_3$ but then $V, W, \sigma V \sigma W\in span\{e_4, e_5, e_6\}$ 
  which gives no basis of $\ct \rr_3^{\CC}$. 
  
  From this contradiction  we get that $\nu\neq 0$ and so $b_4=0$.
 
 From $[V,W]=0$ it may hold $b_2 c_4=0=c_4b_3$. If $c_4\neq 0$ then $b_2=0=b_3$.  Moreover if $\nu\neq 1$ then $b_5=0=b_6$
 which derives in the contradiction $V=0$. Therefore $\nu=1$ implies  $b_6=0$ and so $V\in span\{e_5\}$ which does not allow
 a basis of $\ct \rr_3^{\CC}$. Hence $c_4=0$. If $\nu\neq 1$ then $b_2=0=b_3=b_5=b_6$ thus $V=0$ is a contradiction. 
 Now $\nu=1$ implies $b_3=0$ and $b_6=a_4b_2$ so that $V\in span\{e_2, e_5, e_6\}$. 
 
 From $[U,W]=\mu W$ we have that if $\mu\neq 1$ then $c_2=0=c_3$ and in this way $V, W, \sigma V, \sigma W\in span\{e_2, e_5,
 e_6\}$ which cannot build a basis of $\ct \rr_3^{\CC}$. Thus $\mu=1$ and so $c_3=0$, and therefore the set $V, W, \sigma V, \sigma
 W\in span\{e_2, e_5, e_6\}$  does not induce a basis of $\ct \rr_3^ {\CC}$, implying the non existence of a complex
 subalgebra $\mm$ in this case. 
 
 Assume now $\mm=span\{U,V,W\}$ is a complex subalgebra  such
 that $\ct \rr_3^{\CC}=\mm \oplus \sigma \mm$ and $[V,W]=0$ $[U,V]=\nu V$ 
 $[U,W]=V+\nu W$ for some $\nu \in \CC$.
 
 If $\nu\neq 1$ it follows that $b_2=0=b_3$ which also implies $c_2=0=c_3$. Therefore
 $V,W, \sigma V, \sigma W\in span\{e_4,e_5,e_6\}$ and we have no complex structure. 
 
 Hence $\nu=1$. Thus $b_3=0=b_4=c_4$ and $c_3=b_2$. From $[U,V]=V$ one gets  $b_6+a_4b_2=0$ and from
 $[U,V]=V+W$ one has $b_6-a_4 b_2=0$; therefore $b_6=0=a_4 b_2$. Since $a_4\neq 0$
 then $b_2=0$ but this implies $V=b_5 e_5$ and so $V, \sigma V$ are not linearly
 independent. Hence no complex structure arise in this case.   
\end{proof}




\begin{prop} The following statements are equivalent:

i) $\ct \rr_{3,\lambda}$ can be endowed with a complex structure;

ii) $\ct \rr_{3, \lambda}$ carries an abelian complex structure;

iii) $\lambda=0$.
\end{prop}
\begin{proof} One proves i) $\Longrightarrow$ iii) $\Longrightarrow$ ii)
$\Longrightarrow$ i). It is easy to see ii) $\Longrightarrow$ iii).

We will give a general line for the proof. 

Let $\mm$ denote a complex subalgebra of $\ct \rr_{3,\lambda}^{\CC}$
 spanned by vectors $U,V, W$ as in (\ref{uv}), with $[V,W]=0$.

The Lie brackets follow
$$
\begin{array}{rcl}
[V,W] & = & (b_4 c_2 - b_2 c_4) e_5 + \lambda (b_4 c_3 - b_3 c_4)e_6 \\
\,[U,V] & = & b_2 e_2 + \lambda b_3 e_3 + b_5 e_5 + \lambda b_6 e_6 +(a_4 b_2-a_2b_4) e_5
+ \lambda (a_4 b_3-a_3b_4) e_6\\
\,[U,W] & = & c_2 e_2 + \lambda c_3 e_3 + c_5 e_5 + \lambda c_6 e_6+ (a_4 c_2-a_2 c_4)
e_5 + \lambda (a_4c-3-a-3c_4) e_6
\end{array}
$$

If  the action of $U$ on $span\{V,W\}$ is of type (1)
in (\ref{typ1}), by solving the corresponding system, one gets a basis of $(\ct
\rr_{3, \lambda})^{\CC}$ only if $\lambda =0$, with the additional
constraints $\nu=0=\mu$. Explicitly, the vectors adopt the form
$$
U=e_1 + a_2 e_2 + a_3 e_3+ a_4 e_4+ a_5 e_5 + a_6 e_6,$$
\begin{equation}\label{uvwa}
 V= b_3 e_3 + b_4 e_4 - a_2 b_4 e_5+ b_6 e_6, \qquad  W= c_3 e_3 + c_4 e_4 - a_2 c_4 e_5+ c_6 e_6
 \end{equation}
whenever $U,V,W, \sigma U, \sigma V, \sigma W$ is a basis of $(\ct
\rr_{3,0})^{\CC}$. It follows at once that the induced complex
structure on $\ct \rr_{3,0}$ is abelian.

If the action of $U$ on $span\{V,W\}\simeq \CC^2$ is of type (2)
in (\ref{typ1}), then we cannot find a complex structure for any value of $\lambda$.
This argument shows i) $\Longrightarrow$ iii). 

For ii) $\Longrightarrow$ iii) one works  out the equations deriving from $[V, W]=0=[U,V]=[U,W]$ to obtain that a solution exists only for $\lambda=0$. In this case, one gets  the vectors $U,V,W$ above (\ref{uvwa}). For
instance the following $J$ gives rise to a abelian complex structure on $\ct \rr_{3,0}$:
\begin{equation}\label{abaffr}
J e_1 = e_2 \qquad J e_3 = -e_6 \qquad J e_4 = e_5.
\end{equation}
To prove iii) $\Longrightarrow$ ii) one must solve the equation $[V,W]=0$, $[U,V]=\nu
V$ and $[U,W]=\mu W$ for $\lambda =0$. It is possible to see that the only way to get
solutions is for $\nu=\mu=0$, finishing the proof.
\end{proof}

\subsection{Complex structures on cotangent Lie algebras of dimension six}
 Recall the Lie group counterpart of the cotangent Lie algebra. The zero section
in the cotangent bundle $T^*H$ of a Lie group $H$, can be identified with $H$, as well as the fibre over
$(e,0)$ with $\hh^*$. As a Lie group, the cotangent bundle of $H$
is  the semidirect product of $H$ with $\hh^*$ via the coadjoint
representation. The tangent space of $T^*H$ at the identity is
naturally identified with the {\em cotangent} Lie algebra
$\ct^*\hh:= \hh \ltimes_{\coad} \hh^*$,  the semidirect
product of $\hh$ and its dual $\hh^*$ via the coadjoint action.

\begin{prop} \label{cot3} Let $\hh$ be a solvable real Lie algebra of
dimension three  and let $\ct^* \hh$ denote the cotangent Lie
algebra spanned by $e_1, e_2, e_3, e_4, e_5, e_6$.  The non zero  Lie brackets  are listed below:
$$
\begin{array}{ll}
\ct^*\hh_1: & [e_1,e_2]=e_3, \, [e_1, e_6]=-e_5, [e_2, e_6]=e_4\\
\ct^*\rr_{3}: & [e_1,e_2]=e_2,\, [e_1,e_3]= e_2 + e_3,\, \\
& [e_1, e_5]= -e_5 -e_6,\, [e_1, e_6]=-e_6,\, [e_2, e_5]=e_4,\, [e_3, e_5]=e_4,\, [e_3,e_6]=e_4 \\
\ct^*\rr_{3,\lambda}: & [e_1,e_2]=e_2,\, [e_1,e_3]= \lambda e_3  \\
  |\lambda| \leq 1& [e_1, e_5]= -e_5,\, [e_1, e_6]=-\lambda e_6,\, [e_2, e_5]=e_4,
\,  [e_3,e_6]=\lambda e_4\\
 \ct^* \rr_{3,\eta}': & [e_1,e_2]=\eta e_2- e_3,\, [e_1,e_3]= e_2 + \eta e_3,
 [e_1, e_5]= -\eta e_5-e_6,\,  \\
 \eta \geq 0 &  [e_1, e_6]= e_5-\eta e_6,\, [e_2,e_5]=\eta e_4,\, [e_2,e_6]= -e_4,
 \, [e_3, e_5]= e_4,\, [e_3,e_6]=\eta e_4\\
\end{array}
$$
\end{prop}


\begin{thm} \label{teoct} Let $\hh$ denote a three dimensional solvable real Lie algebra,
 if  $\ct^*\hh$ admits a complex structure then $\hh$ is isomorphic to one of the 
 following Lie algebras: $\hh_1$, $\RR \times \aff(\RR)$, $\rr_{3,1}$, $\rr_{3,-1}$, 
 $\rr_{3, \eta}'$ for  any $\eta \geq 0$.
\end{thm}

The following propositions gives the proof of  the theorem above.

\begin{lem} Let $\hh$ denote a solvable real Lie algebra of dimension three. If $\mm$ 
is a complex subalgebra of $\ct^*\hh$ such that there exist $V,W\in \mm$ satisfying $[V,W]=W$ then $W=0$.
\end{lem}
\begin{proof} Assume $\mm$ is a complex subalgebra of $\ct^*\hh$ with $V,W\in \mm$
satisfying $[V,W]=W$. From the Lie bracket relations in (\ref{cot3}) one has $W= c_4
e_4$ which implies $W=0$. 
\end{proof}

\begin{cor} Let $\hh$ denote a solvable real Lie algebra of dimension three. 
If $(\ct^*\hh)^{\CC}$ splits as a direct sum $(\ct^*\hh)^{\CC}=\mm \oplus \sigma \mm$, 
where $\mm$ is a complex subalgebra and $\sigma$ is the conjugation with respect to $\ct^*\hh$, then $\mm\simeq \CC \ltimes \CC^2$.
\end{cor}

For the existence problem of complex structures it remains the study of complex
subalgebras $\mm$ such that $\mm
\simeq \CC\ltimes \CC^ 2$, we do below.

\begin{prop} \label{coth1} Every complex structure on the Lie algebra  
$\ct^*\hh_1$  is three step nilpotent. 
\end{prop}
\begin{proof} In (\ref{exa1}) we proved that $\ct^*\hh_1$ cannot admit an abelian
complex structure. For the sake of completeness we shall see first
that $\ct^*\hh_1$ has a complex structure. Since $\ct^*\hh_1$ is
nilpotent, if it admits a complex structure, then the
corresponding complex subalgebra $\mm$ must be  nilpotent,  hence
$\mm =span\{U,V,W\}$ with $[V,W]=0$, the action of $U$ of type (2)
in (\ref{typ1}) so that  $[U,V]=0$ and $[U,W]=V$ .  Explicitly

\vskip 3pt

- $[V,W]=0$ implies $b_2 c_6 - b_6 c_2=0,$

\vskip 3pt

- $[U,V]=0$ and $[U,W]=V$ imply the following systems of
equations:
$$\begin{array}{rclrcl}
b_2 & = & 0, & -c_6 & = & b_5,\\
a_2 b_6 - a_6 b_2 & = & 0, & c_2 & = & b_3,\\
a_2 c_6 - a_6 c_2 & = & b_4, &\quad b_6 & = & 0.
\end{array}
$$

Since these systems have solutions,  complex subalgebras of $(\ct^*\hh_1)^{\CC}$ spanned
 by $U,V,W$ of the form
$$
U=e_1 + a_2 e_2 + a_3 e_3 + a_4 e_4 + a_5 e_5 +a_6 e_6,$$
$$\, V = c_2 e_3 + (a_2 c_6 - a_6 c_2) e_4 - c_6 e_5, \qquad W= c_2 e_2 + c_3 e_3 + c_4 e_4 + c_5 e_5 + c_6 e_6$$
induce  complex structures if and only if the vectors $U,V,W, \sigma{U},
\sigma{V}, \sigma{W}$ span a basis of $(\ct^*\hh_1)^{\CC}$.

After \cite{Mg} there is only one class of complex
structures, therefore anyone  is  equivalent to the complex structure 
$J$ given by
\begin{equation}\label{h1cs} J e_1 = e_4 \qquad J e_2 = e_6 \qquad Je_5 = e_3.
\end{equation}

 For this complex structure, notice that
$\aa_0(J)=\{0\}$, $\aa_1(J)=span\{e_3, e_5\}$,
$\aa_2(J)=span\{e_1, e_3, e_4, e_5\}$, $\aa_3(J)=\ggo$. Hence from
lemma (\ref{le}) we conclude that any complex structure on $\ct^*
\hh_1$ is nilpotent.

However notice that there are complex structures which are not totally real, for instance the
next one
\begin{equation}\label{h1notr}
Je_1=e_2-e_4 \qquad J e_2 =e_6 \qquad Je_5= e_3+ e_4.
\end{equation}

\end{proof}

\begin{remark} This Lie algebra is isomorphic to $\mathcal G_{6,3}$ in \cite{Mg} and to $\mathfrak h_7$ in \cite{CFU}.
\end{remark} 

\begin{prop} The Lie algebra $\ct^*\rr_3$ cannot be endowed with a complex structure.
\end{prop}
\begin{proof} Assume $\mm$ is a complex subalgebra of $(\ct^*\rr_3)^{\CC}$
spanned by vectors $U,V,W$ as in (\ref{uv}) and such that $[V,W]=0$.
Assume $U$ acts on $\aa = span\{V,W\}$. 

 Case 1). We assume the action of $U$ on  is of type
 (1) in (\ref{typ1}), then if $[U,V]= \nu V$ and $[U,W]= \mu W$ we
 have that $U= e_1 + a_2 e_2+a_3e_3+a_4 e_4 + a_5 e_5 +a_6
e_6$, $V= b_2 e_2 + b_4 e_4+ b_6 e_6$, $W= c_2 e_2 + c_4 e_4+ c_6
e_6$. It is not difficult to verify that $U,V,W, \sigma{U},
\sigma{V}, \sigma{W}$ cannot be a basis of $(\ct^*\rr_3)^{\CC}$.

\smallskip

Case 2) Suppose that the action of $U$ on $\aa = span\{V,W\}$ is
of type (2). From the constraint $[U,V]= \nu V$  and  $[U,W]=
V + \nu W$, we get that there is no possibility of choosing  $\{U,V,W, \sigma
U, \sigma V, \sigma W\}$ as a linearly independent set in $(\ct^*
\rr_3)^{\CC}$, hence there is no complex structure associated to  such $\mm$.
\end{proof}

\begin{prop} If the Lie algebra $\ct^*\rr_{3, \lambda}$ admits a complex structure then  $\lambda =0, 1, -1$.
\end{prop}
\begin{proof}  Let $\mm$ be a complex subalgebra of $(\ct^*\rr_{3,\lambda})^{\CC}$ spanned by the
vectors $U,V,W$ as in (\ref{uv}) and such that $[V,W]=0$. The
following constraint must be satisfied:
$$b_2 c_5 - b_5 c_2 + (b_3 c_6 - b_6 c_3) \lambda=0.$$

The Lie brackets on $\ct^*\rr_{3, \lambda}$ follow
$$\begin{array}{rcl}
[U,V] & = & b_2 e_2 +\lambda b_3 e_3 + [a_2 b_5-a_5 b_2+\lambda(a_3
b_6-a_6 b_3)] e_4 - b_5 e_5 - \lambda b_6 e_6\\
\, [U,W] & = & c_2 e_2 +\lambda c_3 e_3 + [a_2 c_5 -a_5 c_2 +\lambda(a_3 c_6 -a_6 c_3)]e_4 -c_5
e_5-\lambda c_6 e_6\\
\, [V,W] & = & [b_2 c_5 - b_5 c_2 +\lambda(b_3 c_6- b_6 c_3)] e_4
\end{array}
$$

 Case 1) Assume the action of $U$ on $\aa = span\{V,W\}$ is of type (1) in (\ref{typ1}). The conditions
  $[U,V]= \nu V$ and $[U,W]=\mu W$  shows that a subalgebra
$\mm$ exists if $\lambda\in\left\{0,1,-1\right\}$. Moreover, such
an $\mm$ is spanned by $U,V,W$  as given in the following table:

$$\begin{array}{|c|l|} \hline
 & \quad U = e_1 + a_2 e_2 + a_3 e_3 + a_4 e_4 + a_5 e_5 + a_6 e_6,\\ \quad \lambda=0 \quad
 & \quad V =b_3 e_3 + b_4 e_4 + b_6 e_6, \, W=-a_2 c_5 e_4 + c_5 e_5 \quad  \mbox{ or } \quad \\
& \quad U, \, V \,\mbox{as above and }\, W=c_2 e_2 -a_5 c_2 e_4 \\ \hline
 & \quad U = e_1 + a_2 e_2 + a_3 e_3 + a_4 e_4 + a_5 e_5 + a_6 e_6,\, \\
 & \quad V =b_2 e_2 +b_3 e_3 - (a_5 b_2 + a_6 b_3) e_4, \, \\ \lambda=1 & \quad W=-\frac{c_5}{b_3}
 (a_2 b_3-a_3b_2)+ c_5 e_5 - \frac{b_2c_5}{b_3} e_6\quad \mbox{ with }\quad  b_3 \neq  0 \quad  \mbox{ or } \\
& \quad U \quad \mbox{as above and }\, \\ & \quad  V= -(a_2 b_5 + a_3 b_6) e_4 + b_5 e_5 + b_6 e_6,\,\\
& \quad  W=c_2 e_2 -\frac{b_5 c_2}{b_6} e_3 -\frac{c_2}{b_6}(a_5 b_6 -a_6 b_5)\quad \mbox{ with }
\quad b_6 \neq 0 \\ \hline
 & \quad U = e_1 + a_2 e_2 + a_3 e_3 + a_4 e_4 + a_5 e_5 + a_6 e_6,\, \\ & \quad V
 =-\frac{c_3b_6}{c_5} e_2 -\frac{b_6}{c_5} (a_5 c_3 - a_6 c_5) e_4 + b_6 e_6 \, \\ \lambda=-1 &
  \quad W=c_3 e_3 -(a_2 c_5+a_6c_3) \quad \mbox{ with }\quad  c_5 \neq  0 \quad  \mbox{ or } \\
& \quad U \quad \mbox{as above and }\, \\ &  \quad V= b_3 e_3 -(a_2 b_5 + a_6 b_3) e_4 + b_5 e_5,\,\\
&  \quad W=c_2 e_2 -\frac{c_2}{b_3}(a_3 b_5 -a_5 b_3)e_4-\frac{b_5 c_2}{b_3}e_6
\quad \mbox{ with } \quad b_3 \neq 0 \\ \hline
\end{array}$$

In all cases, $U,V, W, \sigma U, \sigma V, \sigma W$ turn out to
be a basis of $(\ct^* \rr_{3,\lambda})^{\CC}$. We also observe
that none of these complex structure is abelian.

For instance, the linear map on $\ct^*\rr_{3,0}$ given by
\begin{equation}\label{cl0tr}
Je_1=e_5\qquad Je_2=-e_4 \qquad Je_3=e_6
\end{equation}
and such that $J^2=-{\rm I}$ defines a totally real complex structure on the cotangent Lie 
algebra $\ct^*\rr_{3,0}$, while the $J$ taken as
\begin{equation}\label{cl0notr}
Je_1=e_2\qquad Je_4=e_5 \qquad Je_3=e_6
\end{equation}
gives rise to a complex structure which is not totally real. 

For $\lambda=-1$ the linear homomorphism such that $J^2=-{\rm I}$ given by
\begin{equation}\label{cl-1tr}
Je_1=e_4 \qquad Je_2=e_6 \qquad Je_3=-e_5
\end{equation}
gives a totally real complex structure on $\ct^*\rr_{3,-1}$. While the $J$ satisfying $J^2=-{\rm I}$ and 
\begin{equation}\label{cl-1notr}
Je_3=-(e_1+e_6)\qquad Je_5=e_3-e_4 \qquad Je_6=-(e_2+e_4)
\end{equation}
induces a non totally real complex structure on $\ct^*\rr_{3,-1}$.

Now for $\lambda=1$ no complex structure on $\ct^*\rr_{3,1}$ is totally real. For instance
\begin{equation}\label{cl1}
Je_1=e_4\qquad Je_2=e_3 \qquad Je_5=e_6
\end{equation}
is a complex structure on $\ct^*\rr_{3,1}$.

\end{proof}

\begin{prop} The Lie algebra $\ct^*\rr_{3, \eta}'$ carries a  complex structure for any
 $\eta \geq 0$.
\end{prop}
\begin{proof} 
The  linear isomorphisms $J$ such that $J^2=-{\rm I}$ given by
\begin{equation}\label{cseta}
J e_1 = \pm e_4 \qquad Je_2 = e_3\qquad J e_5=e_6
\end{equation}
define
 complex structures on $\ct^*\rr_{3, \eta}'$ for any $\eta\geq 0$.
 
 Notice that on $\ct^*\rr_{3,0}'$ one has totally real complex structures
 (\ref{trct}), for instance 
 \begin{equation}\label{csetatr}
J e_1 = \pm e_4 \qquad Je_2 = e_6\qquad J e_3=-e_5.
\end{equation}

\end{proof}


\section{Complex structures and related geometric structures}
 In these paragraphs  we relate complex structures to some geometric structures.
 We are mainly interested on hermitian structures, symplectic and K\"ahler structures.

\smallskip

\subsection{On Hermitian complex structures.} 

A {\em metric} on a Lie algebra $\ggo$ is a non degenerate symmetric bilinear map, $\la \, , \, \ra:\ggo \times \ggo \to \RR$. It is called {\em ad-invariant} if 
$$\la [x,y],z\ra + \la y, [x,z]\ra=0 \qquad \qquad \mbox{for all }x,y \in \ggo.$$

\begin{example} \label{mecot} The canonical metric on a cotangent Lie algebra $\ct^*\hh$ is defined by 
$$\la (x,\varphi), (x', \varphi')\ra=\varphi'(x)+\varphi(x')\qquad \mbox{ for all }x,x'\in \hh,\,  \varphi, \varphi'\in\hh^*.$$
It is neutral and ad-invariant.
\end{example}

A subspace $\ww \subseteq (\ggo, \la \,,\,\ra)$ is called {\em isotropic} if 
$\la x, y\ra=0$ for all $x,y \in \ww$, that is, if $\ww \subseteq \ww^{\perp}$, where
$$\ww^{\perp}=\{ y\in \ggo \, \mbox{ such that }\, \la x,y \ra =0 \mbox{ for all } x\in \ww\},$$
furthermore $\ww$ is called {\em totally isotropic} whenever $\ww = \ww^{\perp}$.

\begin{example} On $\ct^*\hh$ equipped with its canonical metric, both subspaces $\hh$ and $\hh^*$ are totally isotropic.
\end{example}

Let $(\ggo, \la \, , \, \ra)$ denote a real Lie algebra equipped with a metric. An (almost) complex structure $J$ on $\ggo$ is called {\em Hermitian} if 
\begin{equation} \label{her}
 \la Jx, Jy\ra = \la x, y\ra \qquad \quad \mbox{ for all } x,y\in \ggo.
 \end{equation}
 
 If the metric is positive definite a Hermitian complex structure is also 
 called a  orthogonal complex structure.
 
 Notice that if $J$ is Hermitian, then $\la x, Jx\ra=0$ for all $x\in \ggo$.  The
 non degeneracity property of $\la \, , \, \ra$ says that there is $y\in \ggo$
 such that $\la x, y\ra \neq 0$. Therefore the subspace $\ww= span\{x, Jx, y, Jy\} \subseteq \ggo$ is non degenerate and $J$-invariant. Furthermore 
 $$\ggo=\ww \oplus \ww^{\perp}$$
 is a orthogonal direct sum of $J$-invariant non degenerate subspaces of $\ggo$. 
 A similar argument can be done in the proof of the following lemma.

\begin{lem} Let $\ggo$ denote a real Lie algebra endowed with a metric $\la \, ,\, \ra$ and let $J$ be an almost complex structure on $\ggo$. Assume $\vv$ is a totally real and totally isotropic subspace on $\ggo$, then

i) $\ggo$ admits a decomposition into a direct sum of totally real and totally isotropic vector subspaces 
$$\ggo=\vv \oplus J\vv;$$

ii) $\ggo$ splits into a orthogonal direct sum 
$$\ggo= \ww_1 \oplus \ww_2 \oplus \hdots \oplus \ww_n$$
of $J$-invariant non degenerate subspaces $\ww_1, \hdots, \ww_n$, where $\dim \ww_i\equiv 0$ (mod 4).

\end{lem} 
 
 \begin{cor} Let $\hh$ denote a real Lie algebra, and let $\la\,,\,\ra$ denote a metric on $\ct_{\pi}\hh$ for which $\hh$ is totally isotropic and assume $J$ is a totally real almost complex structure on $\ct_{\pi} \hh$. Then the dimension of $\hh$ is even.
 \end{cor}
 

\begin{defn} A generalized complex structure on a Lie algebra $\hh$ is a
 Hermitian complex structure on $(\ct^*\hh, \la \, , \, \ra)$ where 
 $\la \,,\,\ra$ denotes the canonical metric on $\ct^*\hh$.
\end{defn}

A Hermitian complex structure $J$ on $\ct^*\hh$ which leaves $\hh$ invariant is
called a generalized complex structure of {\em complex type} and it 
corresponds to a complex structure on $\hh$. A Hermitian  complex structure $J$ on 
$\ct^*\hh$ which is  totally  real, that is  $J \hh=\hh^*$,   is said a  generalized complex structure 
 of {\em symplectic type}.  It corresponds to a symplectic structure on $\hh$.

\vskip 3pt
  
Assume  $\hh$ is a Lie algebra which is equipped with an ad-invariant metric $(\,,\,)$.  In this case it can easily be shown that the adjoint and the coadjoint representations are equivalent. In fact if $\ell:\hh \to \hh^{\ast}$ is the linear isomorphism given by $x \to \ell(x)$ such that $\ell(x)y=(x, y)$, it is straighforward to verify that $\ell^{-1} \coad(x) \ell = \ad(x)$ for all $x\in \hh$. 

Results   (\ref{ceq}) and (\ref{le22}) of previous sections imply that  totally real complex structures $J$ on $\ct^*\hh$    correspond to non singular derivations of $\hh$. Explicitly,  a non singular derivation $d$ on $\hh$ induces the map $\ell \circ d: \hh \to \hh^*$ giving  rise to a complex structure on $\ct^*\hh$.

Consider now the canonical neutral metric $\la \, , \,\ra$ on $\ct^{\ast}\hh$ defined by
$$\la (x, \ell_y),(x', \ell_{y'})\ra=(x, y') + (x', y).$$
 Since $\hh$ and $\hh^*$ are isotropic subspaces for $\la \,, \, \ra$, a complex structure $J$ such that $J\hh=\hh^*$ is Hermitian if and only if
 \begin{equation} 
 \la y, Jx\ra =-\la x, J y \ra\qquad \quad \mbox{ for } x,y\in \hh.
 \end{equation}
  Now, since  $J$ is associated  to a non singular map $j:\hh \to \hh^*$, the latter corresponding to a non singular derivation $d$ of $\hh$, we have the linear isomorphism $j:\hh \to \hh^{\ast}$ equals 
  \begin{equation} \label{jj}
  j =\ell \circ d,
  \end{equation}
  thus  both (\ref{her}) and (\ref{jj}) imply
 $$(x, d y) = - (dx, y) \qquad \mbox{ for all } x,y \in \hh$$
 this means that  the Hermitian complex structures on $\ct^* \hh$, such that $J\hh=\hh^*$ correspond to non singular skew symmetric derivations of $(\hh, (\, , \,))$. 
 
  The previous explanations and  \cite{Ja}  derive the following.

 \begin{cor} Let $\hh$ denote an even dimensional Lie algebra endowed with  an ad-invariant metric $(\,,\,)$. The following statements are equivalent
 
 i) $\hh$ admits a generalized complex structure  of symplectic type;
 
 ii) $\hh$ admits a symplectic structure;
 
 iii) $\hh$ admits a non singular derivation which is skew symmetric for $( \,, \,)$.
 
  Furthermore if any of these structures exists  then  $\hh$ is nilpotent.
 \end{cor}

\subsection{Complex structures and symplectic structures}

A {\em symplectic structure} on a even
dimensional Lie algebra $\ggo$  is a  2-form $\omega \in
\Lambda^2(\ggo^{\ast})$ such that $\omega$ has maximal rank, i.e.,
 $\bigwedge^{\frac{1}{2}\dim \ggo}\omega\neq 0$ and it is closed:
 \begin{equation}\label{close}
 \omega([x,y], z)+\omega([y,z],x)+\omega(z,[x,y])=0 \qquad\mbox{for all } x,y, z\in \ggo.
 \end{equation}

 Let ($\ct_{\pi} \hh=\hh \ltimes V, \omega)$ denote a semidirect product equipped with a symplectic structure.  Following \cite{CLP} we say that $\ct_{\pi} \hh$  is {\em lagrangian}, if both $\hh$ and $V$ are lagrangian subspaces relative to  $\omega$. We also say that $\omega$ is {\em lagrangian symplectic}. 
 
 Let $\ct_{\pi} \hh$ denote a generalized tangent Lie algebra, then its dual Lie algebra is the semidirect
product $\ct_{\pi^*} \hh := \hh \ltimes_{\pi^*} V^*$,  where $\pi^*$ is the dual representation
$$(\pi^*(x)a)(u) := -a(\pi(x)(u)) \qquad \qquad x\in \hh, a\in V^*, u\in V.$$
 Note that  the cotangent Lie algebra $\ct^*\hh$ is the dual of the tangent Lie algebra $\ct\hh$. 
 
Suppose $\ct_{\pi} \hh= \hh \ltimes_{\pi} V$  is a Lie algebra equipped with a totally real complex structure $J$ (that is $J\hh=V$). This enables us to define on  $\ct_{\pi^*} \hh:= \hh \ltimes_{\pi^*}V^*$ 
a two-form $\omega_J$ by 
$$\omega_J(x + u, y + v) := v(Jx) - u(Jy),\qquad \quad 
\mbox{where $x, y$ are in $\hh$ and $u, v$ are in ($J\hh=V)^*$}.$$
Then $\omega_J$ is non-degenerate and symplectic since $J$ is integrable (see \cite{BD} or \cite{CLP} for instance). Furthermore the converse is also true,  lagrangian symplectic structures on $\ct_{\pi} \hh$ give rise to  totally real complex structures on $\ct_{\pi^*} \hh$. Therefore

\begin{thm} Totally real complex structures on $\ct_{\pi} \hh=\hh \ltimes_{\pi} V$ 
are in correspondence  to lagrangian symplectic structures on $\ct_{\pi^*} \hh$.
\end{thm}

\begin{cor} 

i) Let $\ct^*\hh$ denote a cotangent Lie algebra. If it admits a lagrangian symplectic structure then $\hh$ is nilpotent.

ii) The tangent Lie algebra $\ct \hh$ admits a lagrangian symplectic structure for any $\hh$ isomorphic to $\hh_1$, $\rr_{3,-1}$, $\rr_{3,0}'$ or $\RR\times \aff(\RR)$.
\end{cor}

If $\ggo$ is a Lie algebra which carries a symplectic structure and a complex
structure, one searches for  a compatible pair
$(\omega, J)$, called a K\"ahler structure, 
\begin{equation}\label{kp}
\omega(Jx,Jy)=\omega(x,y) \qquad \qquad \mbox{ for all } x,y \in \ggo.
\end{equation}

Let $\psi\in Aut(\ggo)$ denote a automorphism of $\ggo$, since
$\wedge^n \psi^* \omega=\psi^*\wedge^n\omega$ and $\psi^*d\omega=d(\psi^*\omega)$, one has
that the existence of a compatible symplectic for a fixed complex
structure $J$ is equivalent to the existence  of a compatible symplectic
structure for every complex structure in the orbit of $J$ under the action of 
$Aut(\ggo)$.

In fact if $J$ is compatible with $\omega$ and $J'=\psi^{-1}J\psi$ then 
 $\omega'=\psi^*\omega$ is compatible with $J'$: 
 $$\omega'(J'x,J'y)=\omega'(\psi^{-1}J\psi x, \psi^{-1}J\psi y)=\omega(J\psi x, J \psi
 y)=\omega(\psi x, \psi y)=\omega'(x,y).$$

\begin{lem} \label{ek} Let $\omega$ denote a two form on $\ggo$ which is compatible with the
complex structure $J$. Let $\psi\in Aut(\ggo)$ be a automorphism such that
$J'=\psi^{-1}J\psi$, then $\psi^*\omega$ is  compatible with $J'$.
\end{lem}

A {\em K\"ahler Lie algebra} is a triple $(\ggo,J,
\omega)$ consisting of a Lie algebra equipped with a K\"ahler structure. The
K\"ahler pair $(J, \omega)$ origines a Hermitian structure on $\ggo$ by defining
a metric $g$ as 
\begin{equation}\label{km}
g(x,y)=\omega(Jx,y) \qquad \qquad \mbox{ for all } x,y \in \ggo. 
\end{equation}
  
  This kind of Hermitian structures satisfies the  parallel condition
  $$\nabla J\equiv 0$$
  where $\nabla$ denotes the Levi Civita connection for $g$. The pair $(J,g)$ is
  called a {\em pseudo-K\"ahler metric} on $\ggo$.

 \begin{remark} A Lie algebra $\hh$ equipped with an ad-invariant metric 
 $\la \, , \, \ra$ cannot carry a complex structure $J$ which is Hermitian 
 and parallel with respect to the Levi Civita connection of 
 $\la \, ,\, \ra$ (see \cite{ABO}). However $\hh$ can admit a pseudo K\"ahler 
 metric if one relaxes the assumption on the metric to be ad-invariant 
 (see (\ref{psch1})).
 \end{remark}
 
  It is our aim to investigate the existence of pseudo K\"ahler metrics on the Lie
  algebras $\ct \hh$ and $\ct^*\hh$ where $\hh$ denotes a three dimensional real
  Lie algebra.

 \vskip 3pt
 
 We denote by $e^{ij\hdots}$ the exterior product  $e^i\wedge e^j \wedge \hdots$,
 being $e^1, \hdots, e^6$ the dual basis of $e_1, \hdots ,e_6$.

 \begin{lem} The following Lie algebras do not carry a symplectic structure:
 
 \vskip 3pt
 
 i) $\ct^*\rr_{3, \lambda}$ for any $\lambda$. 
 
 ii) $\ct^*\rr_{3,\eta}'$ for any $\eta\geq 0$.
 \end{lem}
 
  \begin{proof} The proof follows along the following lines. Let $\alpha_{ij} \in \RR$
 be arbitrary constants and define the generic 2-form on $\ct^*\hh$ 
 \begin{equation}\label{2f}
 \theta=\sum_{i<j} \alpha_{ij} e^{ij} \qquad \qquad i=1,\hdots, 5.
 \end{equation}
 
 If one requires $\theta$ to be closed, the condition $d\theta=0$ generates a system
 depending on the parameters $\alpha_{ij}$.
 
 We examplify here one case. The Maurer-Cartan equations on $\ct^*\rr_{3,\lambda}$ 
 are given by
 $$\begin{array}{rclrclrcl}
 de^1 & = & 0 & de^2 & = & e^{12} & de^3 & = & \lambda e^{13}\\
 de^4 & = & e^{25}+\lambda e^{36} & de^5 & = & -e^{15} & de^6 & = & -\lambda e^{16}
 \end{array}
 $$
 By the expansion of this expression making use of $de^{ij}=de^i \wedge
 e^j-e^i\wedge de^j$ one gets conditions on the parameters $\alpha_{ij}$. For
 instance in the case of the Lie algebras considered in i), one obtains that $\alpha_{1j}=0$ for all $j=1, \hdots
 ,6$, therefore a closed 2-form $\theta$ belongs to $\Lambda^2 \vv^*$ being $\vv=span\{e_1,
 \hdots, e_5\}$, in this implies that $\theta$  cannot be symplectic.
 
 A similar reasonning applies on $\ct^*\rr_{3,\eta}'$.
 \end{proof}

 \vskip 2pt
 
  On $\ct \hh_1$ if a two form $\theta$ as in (\ref{2f}) is closed, the constrains
 may satisfy
 \begin{equation}\label{cc1}
 0=\alpha_{36}=\alpha_{46}=\alpha_{56}\qquad 0=\alpha_{34}+\alpha_{16}=
 \alpha_{35}+\alpha_{26}.
 \end{equation}
 The compatibility condition with the abelian complex structure $J$ given by
 $Je_1=e_2$ $Je_3=-e_6$ $Je_4=e_5$ (\ref{h1abe}) implies
 \begin{equation}\label{cc2}
 \alpha_{34}=\alpha_{56}\qquad \alpha_{35}=-\alpha_{46}.
 \end{equation}
 Now (\ref{cc1}) and (\ref{cc2}) amounts to $\alpha_{i6}=0$ for i=1,2,3,4,5,
 therefore a closed two form cannot be symplectic.  
 According to results in \cite{Mg} any abelian complex structure is equivalent to the previous one. Following a similar argument but now searching for the compatibility condition between $\theta$ and the totally real complex structure $J_s$ (\ref{h1tr}), with s=0,1, one obtains that any such a two form cannot be symplectic. 
 
 Consider now the complex structure $J$ given by
 \begin{equation}\label{h1ka}
 Je_1= 2 e_4\qquad Je_2=-e_5\qquad Je_3=e_6
 \end{equation}
 this is totally real and it is compatible with the following closed two forms 
 $$\theta = a(e^{45}-2e^{12})+be^{14}+c(e^{24}-2e^{15})+d e^{25}+ e(e^{26}+e^{35})+fe^{36}.$$
 For instance the following two forms give rise to K\"ahler pairs
 \begin{equation}\label{ka1}
 \omega= e^{45}-2e^{12}+ \mu e^{36}\qquad \qquad \mu \neq 0,
 \end{equation}
 \begin{equation}\label{ka2}
 \omega=e^{14} + \nu(e^{26}-e^{35}) \qquad \qquad \nu \neq 0.
 \end{equation}
  
 This together with (\ref{ek})  prove the next result.
 
 \begin{lem} \label{psth1} The Lie algebra $\ct \hh_1$ carries several K\"ahler structures. However no 
 abelian structure admits a compatible K\"ahler pair. 
 \end{lem}

 \vskip 2pt
 
 Next we see that the cotangent Lie algebra $\ct^*\hh_1$ possesses several K\"ahler 
  pairs $(J, \theta)$.
 
 Let $\theta$ denote a closed two form on $\ct^*\hh_1$, thus
 $$\theta=\sum_{i<j} \alpha_{ij} e^{ij}\qquad \qquad
 \alpha_{34}=0=\alpha_{35}=\alpha_{45},\quad \alpha_{36}=\alpha_{14}+\alpha_{25}.$$

  Consider the complex structure on $\ct^*\hh_1$ given by $Je_1=e_4$
  $Je_2=e_6$ $Je_3=-e_5$ (\ref{h1cs}). The compatibility condition between $J$ and the closed two form $\theta$ implies the following conditions on the contrains $\alpha_{ij}$:
  $$
  \begin{array}{rclrclrcl}
 \alpha_{12}& = & \alpha_{46} \quad & \alpha_{13} & = & -\alpha_{45} \quad & \alpha_{15} &= &-\alpha_{34} \\
  \alpha_{16}& = & \alpha_{24} & \alpha_{23} & = & \alpha_{56} & \alpha_{25} &= &-\alpha_{36} 
  \end{array}
  $$
 therefore any two form on $\ggo$ which is compatible with $J$ has the form
 $$\omega=a(e^{12}+e^{46}) + b(2e^{14}-e^{25}+e^{36})+c(e^{16}+e^{24})+d(e^{23}+e^{56})+e e^{26}+f e^{35}.$$
  
 \begin{lem} \label{psch1} The free 2-step nilpotent Lie algebra in three generators $\ct^*\hh_1$ admits several K\"ahler structures. 
 \end{lem}

\begin{lem} \label{kaff} The tangent Lie algebra $\ct (\RR\times\aff(\RR))$ carries several K\"ahler
structures.
\end{lem}
\begin{proof} Let $J$ denote the complex structure on $\ct \rr_{3,0}$ given by
$Je_1=e_2$, $Je_3=-e_6$, $Je_4=e_5$ (\ref{abaffr}). Canonical computations show that
this complex structure $J$ is compatible with the symplectic forms
\begin{equation}\label{oabg}
\omega_{\alpha,\beta,\gamma}=\alpha e^{12} +\beta(e^{15}-e^{24})+\gamma e^{36}\qquad
\qquad \beta\gamma\neq 0
\end{equation}

therefore the pairs $(J, \omega_{\alpha,\beta,\gamma})$ amount to K\"ahler pairs on
$\ct \rr_{3,0}$.
\end{proof}

In view of  explanations before, the proof of the  theorem below  is straighforward. 

\begin{thm} Let $\hh$ denote a real Lie algebra of dimension three. 

i) If $\hh$ is solvable and  $\ct\hh$ admits a complex structure then
it carries a K\"ahler structure.

ii) If $\ct^*\hh$ carries a K\"ahler structure then $\hh$ is nilpotent.
\end{thm}

\vskip 3pt

\subsection{On the geometry of some pseudo K\"ahler homogeneous manifolds} The goal is the
study of some geometric features on the homogeneous manifolds arising in the previous paragraphs in Lemmas (\ref{psth1}) (\ref{psch1}) and  (\ref{kaff}). 

In particular we shall see that in the nilpotent case there are flat and non 
flat metrics. It was already proved in \cite{FPS} that  pseudo-K\"ahler metrics on nilmanifolds  are Ricci flat.

\vskip 3pt

Let $G_1$ denotes the simply connected Lie group whose Lie algebra is $\ct \hh_1$. Its underlying differentiable manifold is $\RR^6$ together with the multiplication
$$
\begin{array}{rcl}
(x_1,x_2,x_3,x_4,x_5,x_6) \cdot (y_1,y_2,y_3,y_4,y_5,y_6) & = & (x_1+y_1,x_2+y_2,\\
 & & x_3+y_3+\frac12 (x_1y_2-x_2y_1),\\
& &  x_4+y_4+\frac12 (x_2y_6-x_6y_2),\\
& &  x_5+y_5 + \frac12 (x_6y_1-x_1y_6), x_6+y_6).
 \end{array}
 $$
The left invariant vector fields at $Y=(y_1,y_2,y_3,y_4,y_5,y_6)\in G_1$ are
$$
e_1(Y)  =  \frac{\partial}{\partial x_1} - \frac12 y_2\frac{\partial}{\partial x_3} +\frac12 y_5 \frac{\partial}{\partial x_6} \quad  
 e_2(Y) = \frac{\partial}{\partial x_2} + \frac12 y_1\frac{\partial}{\partial x_3} + \frac12 y_4 \frac{\partial}{\partial x_6} \quad 
 e_3(Y)  =  \frac{\partial}{\partial x_3}
 $$
 $$
 e_4(Y)=\frac{\partial}{\partial x_4}-\frac12 y_2 \frac{\partial}{\partial x_6}  \qquad e_5(Y)=\frac{\partial}{\partial x_5}+\frac12 y_1 \frac{\partial}{\partial x_6}\qquad e_6(Y)= \frac{\partial}{\partial x_6}. $$
 
 They satisfy the Lie bracket relations of $\ct \hh_1$. Let $e^i$  denote the dual basis of $e_i$ for i=1,2,3,4,5,6, and let $\cdot$ denote the symmetric product. The pseudo-K\"ahler metric for (\ref{ka1}) is 
 \begin{equation}\label{pka1}
 g_{\mu}= 2e^1\cdot e^5+e^2\cdot e^4 +\mu(e^3\cdot e^3+e^6\cdot e^6) \qquad \mu\neq 0
 \end{equation}
 while for (\ref{ka2}) is
 \begin{equation}\label{pka2}
 g_{\nu}= 2e^1\cdot e^1+ \frac12 e^4 \cdot e^4 +\nu(e^2\cdot e^3+e^5\cdot e^6)\qquad \nu\neq 0.
 \end{equation}
  
 Let $\nabla^{\mu}$ and $\nabla^{\nu}$ denote the corresponding Levi Civita connections for $g_{\mu}$ and $g_{\nu}$ respectively. From the Koszul formula, for $X=x_i e_i$ one gets
 $$\nabla^{\mu}_X=\frac14 \left(\begin{matrix}
 -\mu x_6 & 0 & 0 & 0 & 0 & -\mu x_1\\
 0 &  2 \mu x_6  & 0 & 0 & 0 & 2\mu x_2\\
 -2 x_2 & 2 x_1 & 0 & 0 & 0 & 0\\
 -2\mu x_3 & 0 & -2\mu x_1 & -2\mu x_6 & 0 & -2\mu x_4\\
 0 & \mu x_3 & \mu x_2 & 0 & \mu x_6 & \mu x_5\\
 -2 x_5 & 2 x_4 & 0 & -2x_2 & 2 x_1 & 0
 \end{matrix}
 \right)
 $$
 $$   
 \nabla^{\nu}_X=\frac12 \left(\begin{matrix}
 0 & x_2 & 0 & 0 & \nu x_5 & 0\\
 0 &  0  & 0 & 0 & 0 & 0\\
 - 2x_2 & 0 & 0 & -x_5 & -x_4 & 0\\
 0 & 2\nu x_5 & 0 & 0 & 2\nu x_2 & 0\\
 0 & 0 & 0 & 0 & 0 & 0\\
 -2 x_5 &  x_4 & 0 & -x_2 & 0 &  0
 \end{matrix}
 \right)
 $$
 Clearly the Lie subgroup $H_1$ of $G_1$ with Lie algebra $\hh_1\subset \ct \hh_1$ is totally geodesic for $g_{\nu}$  for every $\nu$ but it is not totally geodesic for $g_{\mu}$ for any $\mu$. 
 
 Let $R(X,Y):=[\nabla_X,\nabla_Y]-\nabla_{[X,Y]}$ denote the curvature tensor, 
 for $\nabla$ either the Levi Civita connection $\nabla^{\mu}$ or $\nabla^{\nu}$
 . Notice that $\nabla^{\nu}_{[X,Y]}\equiv 0$ for all $X,Y$.  
 
 By computing them one can verify:
 
 \vskip 3pt
 
 $\bullet$ The pseudo K\"ahler metrics $g_{\mu}$ are non flat.
 
 \vskip 2pt
 
 $\bullet$ A pseudo K\"ahler metric $g_{\nu}$ is flat if $\nu=1$.
 
 \begin{prop} The Lie algebra $\ct \hh_1$ carry flat and non flat pseudo K\"ahler metrics.
 \end{prop}

The Lie algebra $\ct^*\hh_1$ is the free 2-step nilpotent Lie algebra in three generators. Its simply connected Lie group $G_2$, lies on $\RR^6$ with the multiplication given by
$$
\begin{array}{rcl}
(x_1,x_2,x_3,x_4,x_5,x_6) \cdot (y_1,y_2,y_3,y_4,y_5,y_6) & = & (x_1+y_1,x_2+y_2,\\
 & & x_3+y_3+\frac12 (x_1y_2-x_2y_1),\\
& &  x_4+y_4+\frac12 (x_2y_6-x_6y_2),\\
& &  x_5+y_5 + \frac12 (x_6y_1-x_1y_6), x_6+y_6).
 \end{array}
 $$
 
 The left invariant vector fields at $Y=(y_1,y_2,y_3,y_4,y_5,y_6)$ are
$$
e_1(Y)  =  \frac{\partial}{\partial x_1} - \frac12 y_2\frac{\partial}{\partial x_3} +\frac12 y_6 \frac{\partial}{\partial x_5} \quad  
 e_2(Y) = \frac{\partial}{\partial x_2} + \frac12 y_1\frac{\partial}{\partial x_3} - \frac12 y_6 \frac{\partial}{\partial x_4} \quad 
 e_3(Y)  =  \frac{\partial}{\partial x_3}
 $$
 $$
 e_4(Y)=\frac{\partial}{\partial x_4}\qquad e_5(Y)=\frac{\partial}{\partial x_5}\qquad e_6(Y)= \frac12 y_2 \frac{\partial}{\partial x_4}-\frac12 y_1 \frac{\partial}{\partial x_5} +\frac{\partial}{\partial x_6}, $$
 and let $e^i$ denote the dual left invariant 1-forms for i=1,2,3,4,5,6. Consider the metric on $G$ given by
 $$g= 2 e^1 \cdot e^1 + e^2\cdot e^3 + 2 e^4\cdot e^4 - e^5\cdot e^6$$
 where $\cdot$ denotes the symmetric product. In particular
 $$\ct^*\hh_1=\hh_1 \oplus J\hh_1$$
 denotes a orthogonal direct sum  as vector spaces of totally real subalgebras.

 The corresponding Levi Civita connection is given by
 $$\nabla_X=\left( \begin{matrix} 0 & \frac{x_2}2 & 0 & 0 & 0 & -\frac{x_6}2\\
 0 & 0 & 0 & 0 & 0 & 0\\
 -x_2 & 0 & 0 & 0 & 0 & -x_4\\
 0 & 0 & 0 & 0& 0 & x_2\\
 x_6 & -x_4 & 0 & x_2 & 0 & 0\\
 0 & 0 & 0& 0& 0& 0
 \end{matrix}
 \right)
 \qquad \mbox{ for } X=\sum x_i e_i.$$
 One can verify that the Lie subgroup $H_1$ that corresponds to the Lie subalgebra $\hh_1 $,  spanned by $e_1,e_2,e_3$, is totally geodesic.
 
 \vskip 2pt
 
 The corresponding curvature tensor $R(X,Y)$ is given by
 $$R(X,Y)Z=(x_2y_6-x_6y_2)(\frac32 z_6 e_3+\frac12 z_2 e_5).$$

\begin{prop} The Lie algebra $\ct^*\hh_1$ admits non flat but Ricci flat pseudo K\"ahler metrics. 
\end{prop}

\

The simply connected Lie group $G_3$ with Lie algebra $\ct \rr_{3,0}$ is, as a manifold,
diffeomorphic to $\RR^6$. Let $(x_1, x_2, \hdots, , x_6)$ denote an arbitrary element in
$\RR^6$, then the rule multiplication is given by
$$
\begin{array}{rcl}
(x_1,x_2,x_3,x_4,x_5,x_6) \cdot (y_1,y_2,y_3,y_4,y_5,y_6) & = & (x_1+y_1,x_2+e^{x_6}y_2,
x_3+e^{x_6}y_3+ \\ & & + \frac{e^{x_6}}2 (x_1y_2-x_2y_1), x_4+y_4, x_5+y_5, \\
& & x_6+y_6).
 \end{array}
 $$
 
 The left invariant vector fields at $Y=(y_1,y_2,y_3,y_4,y_5,y_6)\in G_3$ are
 $$e_1(Y)=\frac{\partial}{\partial x_1} - e^{y_6} y_2\frac{\partial}{\partial x_3}\qquad 
 e_2(Y)=e^{y_6}(\frac{\partial}{\partial x_2} +  y_1\frac{\partial}{\partial x_3})$$
 $$e_3(Y)=e^{y_6}\frac{\partial}{\partial x_3}\qquad e_4(Y)=\frac{\partial}{\partial x_4}\qquad e_5(Y)=\frac{\partial}{\partial x_5}\qquad e_6(Y)=\frac{\partial}{\partial x_6}, $$
 
 and let $e^i$ denote the dual 1-forms for i=1,2,3,4,5,6.
 
 Consider the metric $\la \, , \, \ra$ on $G$ for which the vector fields above satisfy the non zero relations
 $$g = \alpha (e_1\cdot e_1+ e_2\cdot e_2) + \beta (e_1\cdot e_4+ e_2\cdot e_5)+\gamma( e_3\cdot e_3+ e_6\cdot e_6)\quad \beta \gamma \neq 0.$$
 
 This metric is clearly non definite.

 The complex structure on $G$ is defined as the linear map $J:T_Y G \to T_YG$ such that $J^2=-{\rm I}$ and 
 $$J e_1=e_2 \qquad Je_3=e_6 \qquad Je_4=e_5.$$
 
 This gives  a complex structure on $\RR^6$ which is invariant under 
 the action of the Lie group $G$, the action is induced from the multiplication on $G$.
  Moreover the complex structure is Hermitian for the metric above and it is parallel 
  for the corresponding Levi Civita connection $\nabla$, which on the basis of left invariant vector fields is given by
  $$\nabla_X=\left( \begin{matrix} 
  0 & x_2 & 0 & 0 & 0 & 0\\
  -x_2 & 0 & 0 & 0 & 0 & 0\\
  0 &0 & 0 & 0 & 0 & 0 \\
  0 & x_5 & 0 & 0 & x_2 & 0\\
  -x_5 & 0 & 0 & -x_2 & 0 & 0\\
  0 & 0 & 0 & 0 & 0 & 0
  \end{matrix}
  \right) \qquad \qquad \mbox{ for } X=\sum_{i=1}^6 x_i e_i.
  $$
 
 Let $\hh$ denote the involutive distribution spanned by $e_2,e_6,e_5$, then it admits a complementary orthogonal distribution $J\hh$, therefore $T_YG=\hh\oplus J\hh$ as orthogonal direct sum. At the Lie algebra level,  one has the following short exact sequence
 
 $$0 \longrightarrow \hh \longrightarrow \ggo \longrightarrow J\hh \longrightarrow 0.$$
 
Notice that $\hh$ is a abelian ideal while $J\hh$ is a abelian subalgebra. Moreover, {\em the complex structure $J$ is totally real} with respect to this decomposition  and the representation $\pi$ deriving from the adjoint action satisfies the conditions of Corollary (\ref{ca}).

Let $H$ denote the Lie subgroup corresponding to the distribution $\hh $ and $JH$ the Lie subgroup corresponding to $J\hh$, which  is totally geodesic. In fact, making use of the formula for $\nabla$, one verifies 
$$\nabla_X Y\subseteq J\hh, \qquad \nabla_{JX}{JY}\subseteq J\hh,\qquad \mbox{ for } X,Y\in \hh$$
 and since $(J,g)$ is K\"ahler, $\nabla_X JY=J\nabla_X Y$ and $\nabla_{JX} Y=-J\nabla_{JX} JY$ for $X,Y\in \hh$.

\vskip 3pt

The curvature tensor $R$ is given by 
$$R(x,y)=-\nabla_{[x,y]}$$
which implies that $J H$ is flat. 

The Ricci tensor $r$ follows $r(X,Y)=2(x_1y_1 + x_2y_2)$ for $X=\sum x_i e_1$, $Y=\sum y_i e_i$, therefore $G_3$ is neither flat nor Ricci flat.

\

\end{document}